\documentclass[12pt,amstex]{amsart}

\usepackage{mathptmx}

\usepackage{stmaryrd}
\usepackage{enumerate}
\usepackage{epsfig}
\usepackage{amsmath}
\usepackage{amssymb}
\usepackage{amscd}
\usepackage{graphicx}
\usepackage{pstricks}
\usepackage{mathrsfs}
\usepackage{tocvsec2}

\usepackage{color}

\usepackage[colorlinks=true,citecolor=blue]{hyperref}

\input xy
\xyoption{all}

\newtheorem{thm}{Theorem}[section]
\newtheorem{lem}[thm]{Lemma}

\newtheorem{ques}[thm]{Question}
\newtheorem{cor}[thm]{Corollary}
\newtheorem{conj}[thm]{Conjecture}
\theoremstyle{definition}
\newtheorem{de}[thm]{Definition}
\newtheorem{exam}[thm]{Example}
\theoremstyle{remark}
\newtheorem{rem}[thm]{Remark}

\numberwithin{equation}{section}

\def \N {\mathbb N}

\def \Z {\mathbb Z}

\def \E {\mathbb E}

\def \O {\mathcal{O}}

\def \A {\mathcal A}

\def \G {\mathcal{G}}
\def \B {\mathcal B}

\def \X {\mathcal{X}}
\def \Y {\mathcal{Y}}

\def \Ind {{\bf {\rm Ind}}}

\def \M {{\bf M}}
\def \AP {{\bf AP}}

\def \id {{\rm id}}

\def \a {\alpha }

\def \ep {\epsilon}
\def \d {\delta}
\def \D {\Delta}
\def \c {\circ}

\def \cl {{\rm cl}}

\begin{document}

\title[Topological characteristic factors and independence]{Topological characteristic factors and independence along arithmetic progressions}

\author{Fangzhou Cai}

\author{Song Shao}

\address{Wu Wen-Tsun Key Laboratory of Mathematics, USTC, Chinese Academy of Sciences and
	Department of Mathematics, University of Science and Technology of China,
	Hefei, Anhui, 230026, P.R. China.}

\email{cfz@mail.ustc.edu.cn}
\email{songshao@ustc.edu.cn}


\thanks{This research is supported by NNSF of China (11971455, 11571335).}

\date{}
\begin{abstract}
Let $\pi: (X,T)\rightarrow (Y,T)$ be a factor map of topological dynamics and $d\in \N$. $(Y,T)$ is said to be a $d$-step topological characteristic factor if there exists a dense $G_\d$ set $X_0$ of $X$ such that for each $x\in X_0$ the orbit closure $\overline{\O}((x, \ldots,x), T\times T^2\times \ldots \times T^d)$ is $\pi\times \ldots \times \pi$ ($d$ times) saturated. In \cite{G94} Glasner studied the topological characteristic factor for minimal systems. For example, it is shown that for a distal minimal system, its largest distal factor of order $d-1$ is its $d$-step topological characteristic factor. In this paper, we generalize Glasner's work in \cite{G94} to the product system of finitely many minimal systems and give its relative version. To prove these results, we need to deal with $(X,T^m)$ for $m\in \N$.
We will study the structure theorem of $(X,T^m)$. We show that though for a minimal system $(X,T)$ and $m\in \N$, $(X,T^m)$ may not be minimal, but we still can have PI-tower for $(X,T^m)$ and in fact it looks the same as the PI tower of $(X,T)$. We give some applications of the results developed. For example, we show that if a minimal system has no nontrivial independent pair along arithmetic progressions of order $d$, then up to a canonically defined proximal extension, it is PI of order $d$;  if a minimal system $(X,T)$ has a nontrivial $d$-step topological characteristic factor, then there exist ``many'' $\Delta$-transitive sets of order $d$.
\end{abstract}



\maketitle

\tableofcontents \settocdepth{subsection}
	
\section{Introduction}
	
This paper is motivated by the work of Glasner on topological characteristic factors along arithmetic progressions \cite{G94}, which is the topological counterpart of Furstenberg's work on ergodic behavior of diagonal measures and a theorem of Szemer\'{e}di on arithmetic progressions \cite{F77}. In this paper we will give a generalized version of Glasner's work, and give some applications of the results we developed.

\subsection{Characteristic factors along arithmetic progressions}\
\medskip

As said in \cite{G94}, one of the basic problems in both the measure theoretical and the topological theory is roughly the following: given a system $(X, T)$ (ergodic or minimal) and $d\in \N$, describe the most general relation that holds for $(d + 1)$-tuples
$(x, Tx, T^2x,..., T^dx)$ in the product space $X \times X \times\ldots \times  X$ ($d + 1$ times). To illustrate
this, let us see Furstenberg's work in \cite{F77}.

The multiple recurrence theorem says that for a measure preserving system $(X,\X,\mu,T)$, $d \in \N$ and  $A\in \mathcal{X}$ with positive measure, there is some $n\in \N$ such that
	\begin{equation*}
	\mu(A\cap T^{-n}A\cap T^{-2n}A\cap \ldots \cap T^{-dn}A)>0.
	\end{equation*}
To prove the multiple recurrence theorem, which is equivalent to the famous theorem of Szemer\'{e}di \cite{Sz}, Furstenberg introduced the following measure (called {\em Furstenberg joining}) \cite{F77}:
\begin{equation}\label{a2}
  \mu^{(d)}=\lim_{N\to \infty}\frac{1}{N}\sum_{n=1}^{N}(\tau_d)_*^n \mu_\Delta^d,
\end{equation}
where $d\in \N$, $\tau_d=T\times T^2\times \ldots \times T^d$ and $\mu_\Delta^d$ is the {\em diagonal measure} on $X^d$ defined by
$$\int_{X^d}f_1(x_1)f_2(x_2)\ldots f_d(x_d)d \mu^d_\Delta= \int_X f_1(x)f_2(x)\ldots f_d(x)d \mu,$$
for all continuous functions $f_1,\ldots,f_d$. In \cite{F77}, it was shown that
$\mu^{(d)}$ is a conditional product measure in $X^d$ relative to $(Y_{d-1}(X), \mu_{d-1})$, where $\{(Y_n(X),\mu_n)\}_n$ is the distal sequence of $(X,\X,\mu,T)$. That is,
\begin{equation}\label{a3}
  \mu^{(d)}=\int_{Y_{d-1}^d}\mu_{y_1}\times \mu_{y_2}\times \ldots \times \mu_{y_d}d \mu_{d-1}^{(d)} (y_1,y_2,\ldots, y_d),
\end{equation}
where $\mu=\int \mu_{y}\ d\mu_{d-1}(y)$ is the disintegration of $\mu$ with respect to $\mu_{d-1}$, and $\mu_{d-1}^{(d)}$ is defined as (\ref{a2}) on $Y_{d-1}$.

This leads to the notion of ``characteristic factors'', which was first explicitly defined by Furstenberg and Weiss in \cite{FW96}. Let $(X,\X,\mu, T)$ be a measurable system and $(Y,\Y,\mu, T)$ be a factor of $X$. Let $d\in \N$. We say that $Y$ is a {\em characteristic factor for $\tau_d$} of $X$ if for all $f_1,\ldots,f_d\in L^\infty(X,\X,\mu)$,
		\begin{equation*}
		\begin{split}
		\frac{1}{N}\sum_{n=0}^{N-1} & f_1(T^{n}x) f_2(T^{2n}x)\ldots f_d(T^{dn}x)\\
		-  &\frac{1}{N}\sum_{n=0}^{N-1} \E(f_1|\Y)(T^{n}x) \E(f_2|\Y)(T^{2n}x)\ldots \E(f_d|\Y)(T^{dn}x)\to 0, N\to\infty
		\end{split}
		\end{equation*}
in $L^2(X,\X,\mu)$. Thus (\ref{a3}) is equivalent to say that $Y_{d-1}$ is a characteristic factor for $\tau_d$ of $X$. This enables us to give a reduction of the problem on the multiple recurrence to its distal factors.

Recently, Host and Kra \cite{HK05} showed that for an ergodic system $(X,\X,\mu,T)$,  $(d-1)$-step pro-nilsystem is a characteristic factor for $\tau_d$, and using this result they proved that $\displaystyle \lim_{N\to\infty} \frac{1}{N} \sum_{n=0}^{N-1}  f_1(T^{n}x)  f_2(T^{2n}x) \ldots f_d(T^{dn}x)$ exists in $L^2(X,\X,\mu)$ (see also \cite{Z}).

\medskip	

The counterpart of characteristic factors in topological dynamics was first studied by Glasner in \cite{G94}. To state the result we need a notion called saturated subset.
Let $X, Y$ be a set and let $\pi : X\rightarrow Y$ be a	map. A subset $L$ of $X$ is called {\em $\pi$-saturated} if $$\{x\in L: \pi^{-1}(\pi(x))\subseteq L\}=L,$$ i.e. $L=\pi^{-1}(\pi(L))$.

\medskip
	
The following is the definition of topological characteristic factors along arithmetic progressions given by Glasner \cite{G94}.
Let $\pi: (X,T)\rightarrow (Y,T)$ be a factor map of topological systems and $d\in \N$. $(Y,T)$ is said to be a {\em $d$-step topological characteristic factor (along arithmetic progressions) } or {\em topological characteristic factor of order $d$} if there exists a dense $G_\d$	set $X_0$ of $X$ such that for each $x\in X_0$ the orbit closure $$L_x=\overline{\O}(\underbrace{(x, \ldots,x)}_{d \ \text{times}}, \tau_d)$$ is $\underbrace{\pi\times \ldots \times \pi}_{d \ \text{times}}$ saturated. That is, $(x_1,x_2,\ldots, x_d)\in L_x$ if and only if $(x_1',x_2',\ldots, x_d')\in L_x$ whenever for all $i\in \{1,2,\ldots, d\}$,		$\pi(x_i)=\pi(x_i')$.

\medskip

The main result of \cite{G94} is that up to canonically defined proximal extensions, a characteristic family for $\tau_d=T\times T^2\times \ldots \times T^d$ is the family of canonical PI flows of order $d-1$. In particular, if $(X,T)$ is a distal minimal system, then its largest distal factor of order $d-1$ is its topological characteristic factor of order $d$; if $(X,T)$ is a weakly mixing system, then the trivial system is its $d$-step topological characteristic factor for all $d\in \N$. In \cite{CS} the authors study the topological characteristic factors along cubes. It is shown that up to proximal extensions the pro-nilfactors are the topological characteristic factors along cubes of minimal systems. In particular, for a distal minimal system, the maximal $(d-1)$-step pro-nilfactor is a topological characteristic factor along cubes of order $d$.

\medskip
	
Since in ergodic theory Host and Kra showed that $(d-1)$-step pro-nilsystem is a characteristic factor for $\tau_d$, it is natural to conjecture that if $(X,T)$ is a distal minimal system, then its maximal $(d-1)$-step pro-nilfactor is its $d$-step topological characteristic factor along arithmetic progressions. Unlike the work in \cite{CS}, it is more difficult to handle topological characteristic factor along arithmetic progressions than topological characteristic factor along cubes.
In this paper, we are unable to solve this question. Our aim here is to generalize Glasner's results and give some applications.

\medskip

To investigate the dynamical systems under the action $\tau_d=T\times T^2\times \ldots \times T^d$, we have to deal with $(X,T^m)$ for $m\in \N$.
Observe that for a minimal system $(X,T)$, $(X,T^m)$ need not be minimal in general. In any case $(X,T^m)$ will have at most $m$ minimal components and we have to work with each of them separately. Keeping track of these components is a technical complication, and hence to simplify the discussion in \cite{G94} the author assume that $(X,T)$ is totally minimal (i.e. $(X,T^m)$ is minimal for all $m\in \N$) in all proofs involved. In this paper, we will not assume that $(X,T)$ is totally minimal, and hence we have to deal with $(X,T^m)$ for $m\in \N$. We will study the structure theorem of $(X,T^m)$. We show that though for a minimal t.d.s. $(X,T)$ and $m\in \N$, $(X,T^m)$ may not be minimal, but we still can have PI-tower for $(X,T^m)$ and in fact it looks the same as the PI tower of $(X,T)$ (see Theorem \ref{PI-not-totaly-min}). This part is also of independent interest.

The main result of \cite{G94} says that for a minimal system $(X,T)$, up to canonically defined proximal extensions, a characteristic family for $\tau_d=T\times T^2\times \ldots \times T^d$ is the family of canonical PI flows of order $d-1$. We will show this result still holds if $(X,T)=(X_1\times X_2\times \ldots \times X_k, T_1\times T_2\times \ldots \times T_k)$, where $k\in \N$ and $(X_1, T_1), \ldots, (X_k,T_k)$ are minimal systems (see Theorem \ref{main}). Note that in this case $(X,T)$ is not minimal in general,  but it remains lots of good properties of minimal systems. When $k=1$, it is the case in \cite{G94}.
And we will also give the relative version of this result.

\subsection{Independent pair along arithmetic progressions}\
\medskip

One application of our results is to study independence pairs along arithmetic progressions in topological systems.
The notion of \emph{independence} was first introduced in \cite{KL}. It corresponds to a modification of the notion of \emph{interpolator} studied in \cite{GW, HY06, HLSY} and was discussed in depth in
\cite{HLY}.

Let $(X,T)$ be a topological dynamical system and ${\A}=(A_1,\ldots,A_k)$ be a tuple of subsets of $X$. We say that a subset $F\subseteq \Z$ is an {\it	independence set} for ${\A}$ if for any non-empty finite subset $J\subseteq F$ and any $s=(s(j):j\in J) \in \{1,\ldots,k\}^J$ we have $\bigcap_{j\in J}T^{-j}A_{s(j)}\not=\emptyset$. The notion of independence is very useful in the study of dynamical systems. For example, it is shown that $(X,T)$ has positive entropy if and only if there are two disjoint closed sets $U_1,U_2$ such that $(U_1,U_2)$ has an independence set of positive density; $(X,T)$ has positive sequence entropy if and only if there are two disjoint closed sets $U_1,U_2$ such that $(U_1,U_2)$ has arbitrarily large finite independence sets.

In \cite{DDMSY}, pairs with arbitrarily long finite IP-independence sets are systematically investigated. For a finite subset $\{p_1,\ldots,p_m\}$ of $\N$, the {\em finite IP-set generated by $\{p_1,\ldots,p_m\}$} is the set $\{\ep_1 p_1+\ldots+\ep_m p_m: \ep_i\in\{0,1\}, 1\le i\le m\}\setminus \{0\}$. Let $(X,T)$ be a topological system, a pair $(x,y)$ in $X^2$ is an {\em $\Ind_{fip}$-pair} if and only if each	${\A}=(A_{1},A_{2})$, with $A_{1}$ and $A_{2}$ neighborhoods of $x$	and $y$ respectively, has arbitrarily long finite IP-independence sets.  In \cite{DDMSY},
it is proved that any minimal system without nontrivial $\Ind_{fip}$-pairs is an almost one to one extension of its maximal $\infty$-step nilfactor, and each invariant ergodic measure is measure theoretical isomorphic to the Haar measure on some
$\infty$-step nilsystem. In particular, a minimal distal system is
an $\infty$-step nilsystem if and only if it has no nontrivial $\Ind_{fip}$-pairs.

\medskip

Similarly, pairs with arithmetic progressions independence sets are studied in \cite{GHSY} recently. A pair $(x,y)$ in $X^2$ is an {\em $\Ind_{ap}$-pair} if and only if each ${\A}=(A_{1},A_{2})$, with $A_{1}$ and $A_{2}$ neighborhoods of $x$ and $y$ respectively, has arbitrarily long arithmetic progressions independence sets. But $\Ind_{ap}$-pairs are more difficult to handle than $\Ind_{fip}$-pairs. In \cite{GHSY}, it is shown that if a minimal system has no nontrivial $\Ind_{ap}$-pairs, then for each ergodic Borel measure $\mu$,
$(X,\mu,T)$ is measure theoretical isomorphic to the Haar measure on some
$\infty$-step nilsystem. Unlike in \cite{DDMSY}, we do not know whether a minimal system without nontrivial $\Ind_{ap}$-pairs is an almost one to one extension of its maximal $\infty$-step nilfactor. Or for a distal minimal system, if it has no nontrivial $\Ind_{ap}$-pairs, is it an $\infty$-step nilsystem?

\medskip

In this paper we try to investigate this question. We could not give the full answer to questions above. But with the help of results developed in the paper, we have some progress.
We show that if a minimal system has no nontrivial $\Ind_{ap}$-pairs, then it is a PI-system; and if a minimal system has no nontrivial independent pairs along arithmetic progressions of order $d$, then up to canonically defined proximal extensions, it is equal to it's canonical PI flow of order $d$. In particular, for a distal minimal system, if it has no nontrivial independent pairs along arithmetic progressions of order $d$, then it is equal to its largest distal factor of order $d$.

\subsection{$\Delta$-transitivity}\
\medskip

By Glasner's result in \cite{G94}, for a minimal weakly mixing system $(X,T)$, the trivial system is its $d$-step topological characteristic factor for each $d\ge 2$, i.e. there is a dense $G_\d$ set $X_0\subseteq X$ such that for all $x\in X_0$, $(x,x,\ldots,x)\in X^d$ has a dense orbit under the action $T\times T^2\times \ldots \times T^d$. Inspired by this result, one can define a system $(X,T)$ (not necessarily minimal) is {\em $\Delta$-transitive} if for every $d\geq 2$ there is a dense $G_\d$ set $X_0\subseteq X$ such that for all $x\in X_0$, $(x,x,\ldots,x)\in X^d$ has a dense orbit under the action $T\times T^2\times \ldots \times T^d$. It was shown that $\Delta$-transitivity implies weak mixing, but there exist strongly mixing systems which are not $\Delta$-transitive \cite{Mo}.

It is easy to see that a t.d.s. $(X,T)$ is $\Delta$-transitive if and only if for every $d\geq 2$ and non-empty open subsets $U_0,U_1,\ldots,U_d$ of $X$, $N(U_0, U_1,\ldots,U_d)\neq\emptyset$, where
\begin{equation*}
N(U_0,U_1,\ldots,U_d)=\{n\in\N: \bigcap_{i=0}^dT^{-in}U_i\neq\emptyset \}.
\end{equation*}
The local version of $\Delta$-transitivity was studied in \cite{HLYZ}. A closed subset $A$ of $X$ is said to be a {\em $\Delta$-transitive set} if for every $d\geq 2$ and non-empty open subsets $U_0,U_1,\ldots,U_d$ of $X$ intersecting $A$, $N(U_0\cap A, U_1,\ldots,U_d)\neq\emptyset$.
The property of $\Delta$-transitive sets has been systematically studied, and refer to \cite{HLYZ,KO,Mo} etc. for details.

\medskip

Similarly, for $d\geq 2$, we say a t.d.s $(X,T)$ is {\em $\Delta$-transitive of order d} if any non-empty open subsets $U_0,U_1,\ldots,U_d$ of $X$, $N(U_0, U_1,\ldots,U_d)\neq\emptyset$.
A closed subset $A$ of $X$ is a {\em $\Delta$-transitive set of order d} if any non-empty open subsets $U_0,U_1,\ldots,U_d$ of $X$ intersecting $A$, $N(U_0\cap A, U_1,\ldots,U_d)\neq\emptyset$.

\medskip

In this paper, we show that using the notion of $\Ind_{ap}$-pairs, we can characterize $\Delta$-transitivity: $(X,T)$ is $\Delta$-transitive if and only if each pair in $X\times X$ is an $\Ind_{ap}$-pair. We also show that if $(X,T)$ has a nontrivial $d$-step topological characteristic factor, then there exist ``many'' $\Delta$-transitive sets of order $d$. To be precise, we show that for a nontrivial factor map $\pi:(X,T)\rightarrow (Y,T)$ of minimal topological dynamical systems and $d\ge 2$, if $(Y,T)$ is a $d$-step topological characteristic factor of $(X,T)$, then there exists a dense $G_\delta$ subset $\Omega$ of $X$ such that for any $x\in\Omega, \pi^{-1}(\pi(x))$ is a $\Delta$-transitive set of order $d$ in $(X,T)$.
Thus if $\pi:(X,T)\rightarrow (Y,T)$ is a nontrivial RIC weakly mixing extension of minimal topological dynamical systems, then there exists a dense $G_\delta$ subset $\Omega$ of $X$ such that for any $x\in\Omega, \pi^{-1}(\pi(x))$ is a $\Delta$-transitive set in $(X,T)$. Also for a distal minimal t.d.s. $(X,T)$ and $d \in \N$, if the order of $(X,T)$ is greater than $d$,
there exists a dense $G_\delta$ subset $\Omega$ of $X$ such that for any $x\in\Omega, \pi^{-1}(\pi(x))$ is a $\Delta$-transitive set of order $d+1$ in $(X,T)$.

\subsection{Organization of the paper}\
\medskip

We organize the paper as follows. In Section \ref{section-pre}, we give some basic notions in topological dynamics. Since our proofs rely heavily on the structure theorem of minimal systems, we give some basic facts on abstract topological dynamics in Section \ref{section-tds}. Since for a minimal system $(X,T)$ and $m\geq 2$, the system $(X,T^m)$ may not be minimal, we need to take care of this very carefully as we are studying the transformation $T\times T^2\times \ldots \times T^d$. Thus in Section \ref{section-decomposition}, we study the decomposition of $T^m$ and in Section \ref{section-PI-tower}, we give the PI-tower for $T^m$. After that we present our generalization of Glasner's theorem in Section \ref{section-Main}. In the next two sections we give applications of our results.

\bigskip

\noindent{\bf Acknowledgments:}
We would like to thank Wen Huang, Hanfeng Li,  Xiangdong Ye and Tao Yu for their very useful comments.

\section{Preliminaries}\label{section-pre}

In this section we introduce some basic notions in topological dynamics.
In this paper, we will consider not only a single transformation $T$,
but also commuting transformations $T_1$, $\ldots$ , $T_k$ of $X$. We only give some basic
definitions and properties for
one transformation, and extensions to the general case are
straightforward.

Integers and natural numbers are denoted by $\Z$ and $\N$ respectively.

\subsection{Topological dynamical systems}\
\medskip
	
By a {\it topological dynamical system} (t.d.s. or system for short) we mean a pair $(X,T)$ where $X$ is a compact metric space (with metric $\rho$) and $T:X \to X$ is a homeomorphism.  For $n \geq 2$ we
write $(X^n,T^{(n)})$ for the $n$-fold product system $(X\times	\ldots \times X,T\times \ldots \times T)$. The diagonal of $X^n$ is $$\Delta_n(X)=\{(x,\ldots,x)\in X^n: x\in X\}.$$
When $n=2$ we write	$\Delta(X)=\Delta_2(X)$.
	
Let $x\in X$. The {\it orbit of $x \in X$} is given by $\O(x,T)=\{T^nx: n\in \Z\}$.	For convenience, sometimes we denote the orbit closure of $x\in X$ under $T$ by  $\overline{\O}(x,T)$ or $\overline{\O}(x)$, instead of $\overline{\O(x,T)}$. Let $A\subseteq X$, the  orbit of $A$ is given by $\O(A,T)=\{T^nx: x\in A,n\in \Z\}$, and $\overline{\O}(A,T)= \overline{\O(A,T)}$.

\medskip

A subset $A\subseteq X$ is called {\em invariant} if $TA= A$. When $Y\subseteq X$ is a closed and
$T$-invariant subset of the system $(X, T)$ we say that the system
$(Y, T|_Y)$ is a {\em subsystem} of $(X, T)$. Usually we will omit the subscript, and denote $(Y, T|_Y)$ by $(Y,T)$. It is obvious that for any $A\subseteq X$, $(\overline{\O}(A,T), T)$ is a subsystem.

	\medskip
	
A t.d.s. $(X,T)$ is {\it transitive} if for any two non-empty open sets $U$ and $V$ there is $n\in \Z$ such that $U\cap T^{-n}V\neq	\emptyset$. A t.d.s. $(X,T)$ is {\em point transitive} if there is some point $x\in X$ with $X=\overline{\O}(x,T)$. Such a point is called a {\em transitive point} of $(X,T)$.
Under our assumptions on $(X,T)$, a t.d.s. is transitive if and only if it is point transitive, and in this case the set of transitive points forms a dense $G_\d$ set of $X$.

A t.d.s. $(X,T)$ is {\it weakly mixing} if the product system $(X\times X,T\times T)$ is transitive. A topological dynamical system $(X,T)$ is {\it minimal} if	$\O(x,T)$ is dense in $X$ for every $x\in X$. $(X,T)$ is \textit{totally minimal} if for every $n\geq 1$, the system $(X,T^n)$ is minimal. A point $x \in X $ is	{\it minimal} or {\it almost periodic} if the subsystem $(\overline{\O}(x,T),T)$ is minimal.
	
	\medskip
	
Let $(X,T)$ and $(Y,S)$ be two minimal t.d.s. and let $(x,y)$ be a minimal point in the product system $(X\times Y,T\times S)$. We denote the minimal system $\overline\O((x,y),T\times S)$ by $(X,x,T)\vee(Y,y,S)$. If $x,y$ are fixed and there is no room for confusion, we sometimes write $(X,T)\vee(Y,S)$ or $X\vee Y$ for short.

\subsection{Factors and extensions}\
\medskip

Let $(X,T)$ and $(Y, S)$ be t.d.s. A {\em factor map} or {\em homomorphism} $\pi: (X, T)\rightarrow (Y,S)$ is a continuous onto map which intertwines the actions (i.e. $\pi\circ T=S \circ \pi$):
\[
\begin{CD}
X @>{T}>> Y\\
@V{\pi}VV      @VV{\pi}V\\
Y @>{S }>> Y.
\end{CD}
\]
We say that $(Y, S)$ is a {\it factor} of $(X,T)$ and that $(X,T)$ is an {\it extension} of $(Y,S)$.
In this paper, we often {\em  use the same symbol for the transformations in both extension and factor, i.e. we use $(Y, T)$ instead of $(Y,S)$}.

\medskip

Let $(X,T)$ be a t.d.s. and let $R$ be a closed invariant equivalence relation on $X$. The quotient space defined by $Y=X/R$ is a compact metric space and if we endow it with the action induced by $T$, we get a t.d.s. $(Y,T)$. By definition, the quotient map $\pi$ from $X$ to $Y$ defines a factor map.
	
On the other hand, let $\pi: (X,T)\rightarrow (Y,T)$ be a factor map. Then
$$R_\pi=\{(x_1,x_2):\pi(x_1)=\pi(x_2)\}$$
is a closed invariant equivalence relation, and $Y=X/ R_\pi$.

For $n\geq 2,$ we write
	\begin{equation*}
		R_\pi^{n}=\{(x_1,x_2,\ldots,x_n):\pi(x_1)=\pi(x_2)=\cdots=\pi(x_n)\}.
	\end{equation*}
It is clear that
$$R_\pi^n=\bigcup_{y\in Y} (\pi^{-1}(y))^n.$$

\subsection{Proximal, distal and regionally proximal relations}\
\medskip

Let $(X,T)$ be a t.d.s. Let $(x,y)\in X\times X$. A pair $(x,y)$ is a {\it proximal}
pair if $$\inf\limits_{n\in \Z} \rho (T^nx, T^ny)=0;$$ it is a {\it distal} pair if it is not proximal. Denote by $P(X,T)$ the set of proximal pairs of $(X,T)$. A t.d.s. $(X,T)$ is {\it equicontinuous} if for every $\epsilon>0$, there exists $\delta>0$ such that $\rho(x,y)< \delta$ implies $\rho(T^nx,T^ny)<\epsilon$ for every $n\in \Z$. A t.d.s. $(X,T)$ is {\it distal} if $P(X,T)= \Delta(X)$.

\medskip

Let $(X,T)$ be a t.d.s. The {\em regionally proximal relation $Q(X,T)$} is defined as: $(x,y)\in Q(X,T)$ if there are sequences $\{x_i\}_{i\in \N},\{y_i\}_{i\in\N}\subseteq X, \{n_i\}_{i\in\N}\subseteq\Z$ and some $z\in X$ such that $(x_i,y_i)\to (x,y)$ and $(T\times T)^{n_i}(x_i,y_i)\to (z,z)$, $i\to \infty$. It is well known that for a minimal system $(X,T)$, $Q(X,T)$ is
an invariant closed equivalence relation and this relation
defines the {\em maximal equicontinuous factor} $X_{eq}=X/Q(X,T)$ of
$(X,T)$. A t.d.s. $(X,T)$ is equicontinuous if and only if $Q(X,T)=\Delta(X)$.


\subsection{Some fundamental extensions}\

\medskip

Let $(X,T)$ and $(Y,T)$ be topological dynamical systems and let $\pi: X \to Y$ be a factor map.
One says that $\pi$ is an {\it open} extension if it is open as a map. $\pi$ is  {\it proximal} if
$\pi(x_1)=\pi(x_2)$ implies $(x_1,x_2) \in P(X,T)$, and $\pi$ is   {\it distal} if $\pi(x_1)=\pi(x_2)$ and $x_1\neq x_2$ implies $(x_1,x_2) \not \in  P(X,T)$.
$\pi$ is an {\it equicontinuous or isometric} extension if for any $\epsilon >0$ there exists $\delta>0$ such that $\pi(x_1)=\pi(x_2)$ and $\rho (x_1,x_2)<\delta$ imply $\rho (T^n(x_1),T^n(x_2))<\epsilon$ for any $n\in \Z$.

\medskip

A factor map $\pi$ is called a {\em weakly mixing extension} if $(R_\pi, T\times T)$ as a subsystem of the product system $(X\times X, T\times T)$ is transitive.

\medskip


Let $\pi: (X,T)\rightarrow (Y,T)$ be a factor map. Call $$P_\pi(X,T)=P(X,T)\cap R_\pi$$
the {\em proximal relation relative to $\pi$}.
We define the {\em relative regionally proximal relation} $Q_\pi(X,T)$ as follows: $(x,y)\in Q_\pi(X,T)$ if and only if there are sequences $\{x_i\}_{i\in \N},\{y_i\}_{i\in\N}\subseteq X, \{n_i\}_{i\in\N}\subseteq\Z$ and some $z\in X$ such that $(x_i,y_i)\in R_\pi$, $(x_i, y_i)\to (x,y)$, and $(T\times T)^{n_i}(x_i,y_i)\to (z,z)$, $i\to \infty$. A factor map $\pi$ is proximal if and only if $P_\pi(X,T)=\Delta(X)$, and $\pi$ is equicontinuous if and only if $Q_\pi(X,T)=\Delta(X)$ (see \cite[Chapter 7, Propositoin 2]{Au88} for example).

\section{Basic facts about abstract topological
	dynamics}\label{section-tds}

In this section we recall some basic definitions and results in
abstract topological systems. For more details, see \cite{Au88, Ellis,
	G76, G96, V77, vries}.

\subsection{Enveloping semigroups}\
\medskip

Given a t.d.s. $(X, T)$ its {\em enveloping semigroup} or {\em Ellis semigroup} $E(X,T)$ is defined as the closure of the set $\{T^n: n\in \Z\}$ in $X^X$ (with its compact, usually non-metrizable, pointwise
convergence topology). For an enveloping semigroup, $E\rightarrow
E:$ $p\mapsto pq$ and $p\mapsto Tp$ is continuous for all $q\in E$. Let $T: X^X\rightarrow X^X$ be defined as $f\mapsto f\circ T$. Then $(X^X,T)$ is a t.d.s. and $E(X,T)$ is its subsystem.

\medskip

Let $(X, T),(Y,T )$ be t.d.s. and $\pi: X\rightarrow Y$ be a factor map. Then there is a unique continuous semigroup homomorphism
$\pi^* : E(X,T)\rightarrow E(Y,T)$ such that
$\pi(px)=\pi^*(p)\pi(x)$ for all $x\in X,p\in E(X,T)$. When there
is no confusion, we usually regard the enveloping semigroup of $X$
as acting on $Y$: $p\pi(x)=\pi(px)$ for $x\in X$ and $p\in E(X,T)$.

\medskip

For a semigroup the element $u$ with $u^2=u$ is called an {\it idempotent}. Ellis-Namakura Theorem says that for a compact Hausdorff space $F$ with a semigroup structure for which right multiplication is continuous, the set $J(F)$ of idempotents of $F$ is not empty (see \cite[Chapter 6, Lemma 6]{Au88}). A non-empty subset $I \subseteq E$  is a {\em left ideal} if $EI \subseteq I$. A {\em minimal left ideal} is a left ideal that
does not contain any proper left ideal of $E$.
By Ellis-Namakura Theorem, all minimal left ideals in $E(X,T)$ contain
idempotents.

\medskip

\subsection{Some facts about universal minimal actions}\
\medskip

Let $\beta\Z$ be the Stone-C\v{e}ch
compactification of $\Z$, which is a compact Hausdorff topological
space where $\Z$ is densely and equivariantly embedded. Moreover,
the addition on $\Z$ can be extended to an addition on $\beta\Z$
in such a way that $\beta\Z$ is a closed semigroup with continuous
right translations. The action of $\Z$ on $\beta\Z$ is point
transitive.

Let $({\bf M},\Z)$ be a minimal left ideal in $\beta\Z$. Then ${\bf M}$ is a closed semigroup with continuous right translations.
By the
Ellis-Namakura Theorem
the set $J=J({\bf M})$ of idempotents in ${\bf M}$ is non-empty. Moreover,
$\{v{\bf M}:v \in J \}$ is a partition of ${\bf M}$ and every $v{\bf M}$ is a group
with unit element $v$.

Let $(X,T)$ be a t.d.s. and then $\Z$ acts on the compact metric space $X$ as follows:
for any $m\in \Z$ and $x \in X$ one has $mx=T^mx$. Then the sets
$\beta\Z$ and ${\bf M}$ also act on $X$ as semigroups and $\beta\Z
x=\{px:p\in \beta\Z\}=\overline{\O}(x,T)$. If $(X,T)$ is minimal, then
${\bf M}x=\overline{\O}(x,T)=X$ for every $x\in X$. For $x\in X$ define
$$J_x=\{v\in J: vx=x\}.$$ It holds that $x$ is minimal if and only
if $J_x$ is not empty. That is,  a necessary and sufficient
condition for $x$ to be minimal is that $ux=x$ for some $u\in J$. Observe that for any invariant closed
subset $A$ of $X$, $JA$ is the collection of minimal points in
$A$.

Let $2^X$ be the collection of non-empty closed subsets of $X$
endowed with the Hausdorff topology. A basis for this
topology on $2^X$ is given by
$$\langle U_1,\ldots,U_n\rangle=\{A\in 2^X: A\subseteq \bigcup_{i=1}^n U_i\
\text{and $A\cap U_i\neq \emptyset$ for every $i\in \{
1,\ldots,n\}$}\},$$ where each $U_i \subseteq X$ is open. The
action of $\Z$ on $2^X$ is given by $mA=\{ma:a\in A\}=\{T^m a: a\in A\}$ for each $m \in
\Z$ and $A \in 2^X$. This action induces another one of $\beta\Z$
on $2^X$. To avoid ambiguities one denotes the action of $\beta\Z$
on $2^X$ by the {\em circle operation} as follows: let $p\in
\beta\Z$ and $A\in 2^X$, and define $p \c
A={\lim\limits_{\lambda}} \ m_\lambda A$ for any net $\{m_\lambda\}_{\lambda \in \Lambda}$ converging to $p$. Moreover
\begin{equation*}
p\c A=\{x\in X: \text{for each $\lambda \in \Lambda$ there is
$d_\lambda\in A$ with $x=\lim_\lambda m_\lambda d_\lambda$}\}
\end{equation*}
for any fixed net  $\{m_\lambda\}_{\lambda \in \Lambda}$ converging to
$p$. Observe that $pA \subseteq p \circ A$, where $pA=\{pa: a \in A\}$.

The following is a well known fact about open mappings (see \cite[Appendix A.8]{vries} for example).

\begin{thm}
	Let $\pi:(X,T)\rightarrow(Y,T)$ be a factor map of t.d.s. Then the map
$	\pi^{-1}:Y\rightarrow 2^X, y\mapsto \pi^{-1}(y)$
	is continuous if and only if $\pi$ is open.
\end{thm}

\subsection{Ellis group}\
\medskip

The group of automorphisms of $(\M,T)$, $G = {\rm Aut} (\M,T)$ can
be identified with any one of the groups $u\M$ ($u\in J$) as
follows: with $\a\in u\M$ we associate the automorphism $\hat{\a}:
(\M,T )\rightarrow (\M,T)$ given by right multiplication
$\hat{\a}(p)=p\a, p\in \M$. The group $G$ plays a central role in
the algebraic theory of minimal systems. It carries a natural $T_1$ compact topology,
called by Ellis the {\em $\tau$-topology}, which is weaker than the
relative topology induced on $G = u\M$ as a subset of $\M$.

\medskip

It is convenient to fix an idempotent $u\in  \M$. As explained above we identify $G$ with $u\M$
and for any subset $A \subseteq G$, {\em $\tau$-topology} is
determined by
$$\cl_\tau A=u(u\c A)=G\cap (u\c A).$$
For a point $x_0\in uX=\{x\in X: ux=x\}$ we associate
a $\tau$-closed subgroup $$\G(X,x_0)=\{\a\in G: \a
x_0=x_0\},$$ which is called the {\em Ellis group} of the pointed system $(X, x_0)$.

For a homomorphism $\pi: X\rightarrow Y$ with $\pi(x_0)=y_0$ we have
$$\G(X, x_0)\subseteq\G(Y, y_0).$$
It is easy to see that $u\pi^{-1}(y_0)=\G(Y, y_0) x_0$.

\medskip

Ellis group can be used to characterize proximal extension (see \cite[Theorem 10.1]{Au88} for example):
\begin{thm}
		Let $\pi:(X,T)\rightarrow(Y,T)$ be a factor map of minimal systems. Let $u\in J(M)$ and $x\in uX,\ y=\pi(x)\in uY$. Then $\pi$ is a proximal extension if and only if $\G(X,x)=\G(Y,y)$.
\end{thm}

\medskip

For a $\tau$-closed subgroup $F$ of $G$ the derived group $H(F)$
is given by:
$$H(F)= \bigcap\big\{ \cl_\tau O : O \ \text{is a
	$\tau$-open neighborhood of $u$ in $F$ }\}.$$ $H(F)$ is a
$\tau$-closed normal subgroup of $F$ and it is characterized as the
smallest $\tau$-closed subgroup $H$ of $F$ such that $F/H$ is a
compact Hausdorff topological group.

\subsection{Furstenberg's structure theorem for distal extensions}\
\medskip

Furstenberg's structure theorem for distal systems \cite{F63} says that any distal minimal system can be constructed by equicontinuous extensions. We state the result in its relative version, which was first given by Ellis \cite{Ellis}. Let $\pi: (X,T)\rightarrow (Y,T)$ be a distal extension of minimal systems. Then there is an ordinal $\eta$
(which is countable when $X$ is metrizable) and a family of systems
$\{(Z_n,T)\}_{n\le\eta}$ such that
\begin{enumerate}
  \item[(i)] $Z_0=Y$,
  \item[(ii)] for every $n <\eta$ there exists a homomorphism
$\rho_{n+1} :Z_{n+1}\to Z_{n}$ which is equicontinuous,
  \item[(iii)] for a limit ordinal $\nu\le\eta$ the system $Z_\nu$
is the inverse limit of the systems $\{Z_\iota\}_{\iota<\nu}$,
\item[(iv)] $Z_\eta=X$.
\end{enumerate}
\begin{equation}\label{a5}
  Y= Z_0  \stackrel{\rho_1} \longleftarrow  Z_1   \stackrel{\rho_2}\longleftarrow \cdots  \stackrel{\rho_n} \longleftarrow Z_n  \stackrel{\rho_{n+1}} \longleftarrow Z_{n+1}
 \longleftarrow \cdots \longleftarrow Z_\eta=X.
\end{equation}

When $Y=\{pt\}$ is the trivial system, we have the structure theorem for a distal minimal system.

Note that the ordinal $\eta$ in \eqref{a5} is not unique. We will discuss this later.


\subsection{PI-systems and PI-extensions}\
\medskip

Let $\pi: (X,T)\rightarrow (Y,T)$ be a factor map of minimal t.d.s.
We say that $\pi$ is a
{\em strictly PI extension} (P is for proximal, and I is for isometric) if there is an ordinal $\eta$
(which is countable when $X$ is metrizable)
and a family of systems
$\{(W_n,w_n)\}_{n\le\eta}$
such that
\begin{enumerate}
  \item[(i)] $Y=W_0$,
  \item[(ii)] for every $n <\eta$ there exists a homomorphism
$\rho_{n+1}:W_{n+1}\to W_n$ which is
either proximal or equicontinuous,
  \item[(iii)] for a limit ordinal $\nu\le\eta$ the system $W_\nu$
is the inverse limit of the systems
$\{W_\iota\}_{\iota<\nu}$,
\item[(iv)] $W_\eta=X$.
\end{enumerate}

We say that $\pi : (X,T)\rightarrow (Y,T)$ is a {\em PI extension} if there
exists a minimal system $(\tilde{X},T)$ such that $\theta : \tilde{X}\rightarrow X$ is proximal homomorphism and $\pi': \tilde{X}\rightarrow Y$ is strictly PI.
$$\xymatrix{
  X \ar[d]_{\pi} & \tilde{X} \ar[l]_{\theta} \ar[dl]^{\pi'}     \\
  Y }
  $$
\medskip

When $Y=\{pt\}$ is the trivial system, we call $(X,T)$ is a {\em PI system}.

\subsection{Structure of minimal systems}\
\medskip

Let $\pi: (X,T )\rightarrow (Y,T)$ be a factor map of minimal t.d.s., and $x_0\in X$, $y_0=\pi(x_0)$. We say that $\pi$
is a {\em RIC} (relatively incontractible) extension if for every $y
= py_0\in Y$, $p\in\M$,
$$\pi^{-1}(y)=p\c u\pi^{-1}(y_0)=p\c Fx_0,$$ where $F =
\G(Y, y_0)$. One can show that $\pi : X \to
Y$ is RIC if and only if it is open and for every $n \ge 1$ the
minimal points are dense in the relation $R^n_\pi$. Note that every
distal extension is RIC, and every distal extension
is open.

\medskip

Every factor map between minimal systems can be
lifted to a RIC extension by proximal extensions (see \cite{EGS} or \cite[Chapter VI]{vries}).

\begin{thm}\label{RIC}
Given a factor map $\pi:X\rightarrow Y$ of minimal systems, there exists a commutative diagram of factor maps (called {\em RIC-diagram} or {\em EGS-diagram}\footnote{EGS stands for Ellis, Glasner and Shapiro \cite{EGS}.})
	
	\[
	\begin{CD}
	X @<{\theta '}<< X'\\
	@VV{\pi}V      @VV{\pi'}V\\
	Y @<{\theta}<< Y'
	\end{CD}
	\]
	such that:
	\begin{enumerate}
		\item[(a)] $\theta '$ and $\theta$ are proximal extensions;
		\item[(b)] $\pi '$ is a RIC extension;
		\item[(c)] $X '$ is the unique minimal set in $R_{\pi \theta
		}=\{(x,y)\in X\times Y ': \pi(x)=\theta(y)\}$, and $\theta '$ and
		$\pi '$ are the restrictions to $X '$ of the projections of $X\times
		Y '$ onto $X$ and $Y '$ respectively.
	\end{enumerate}
\end{thm}

We sketch the construction of these factors. Let $x\in
X$, $u\in J_{x}$ and $y=\pi (x)$. Let $y'=u \c u \pi^{-1}(y)$,
then $y'$ is a minimal point in $2^X$ for the action of $\Z$.
Define $Y'=\{p\c u y': p\in M\}$ to be the orbit closure of $y'$
and $X'=\{(px, p\c y')\in X\times Y': p\in M\}$, and factor maps
given by $\theta (p\c y')= p y$ and $\theta'((px,p\c y')) =px$. It
can be proved that $X'=\{(\tilde x,\tilde y)\in X \times Y' :
\tilde x \in \tilde y\}$.

\medskip

If $\pi:(X,T)\rightarrow (Y,T)$ is a RIC factor map of minimal systems, then $Q_\pi(X,T)$ is
an invariant closed equivalence relation and this relation
defines the {\em maximal equicontinuous factor relative to $\pi$},  $X^{eq}_\pi=X/Q_\pi(X,T)$ of
$(X,T)$. $\pi$ is weakly mixing if and only if $Q_\pi(X,T)=R_\pi$ \cite{V77}.

\begin{thm}\cite{EGS}\label{Thm-EGS}
	Let $\pi:(X,T)\rightarrow (Y,T)$ be a RIC extension of minimal systems with $x_0\in X$ and $y_0=\pi(x_0)\in Y$. Denote by $A,F$ the Ellis group of $X$ and $Y$ respectively, i.e. $A=\G(X,x_0), F=\G(Y,y_0)$. Then there exist a minimal system $(Z,T)$ and factor maps $\rho:(Z,T)\rightarrow(Y,T)$, $\sigma:(X, T)\rightarrow(Z, T)$ such that $\rho\circ\sigma=\pi$, i.e. the following diagram is communicative:
		\begin{equation*}
	\xymatrix{
		X\ar[d]_{\pi}\ar[dr]^{\sigma} \\
		Y  & Z\ar[l]^{\rho}	,
	}
	\end{equation*}
and $\rho$ is an equicontinuous extension and the Ellis group of $Z$ is $H(F)A$, i.e. $B=\G(Z,z_0)=H(F)A$, where $z_0=\sigma (x_0)$.

The system $(Z,T)$ is the largest equicontinuous extension of $Y$ within $X$, and
$\rho$ is an isomorphism if and only if  $H(F)A=F$ if and only if $\pi$ is a weakly mixing extension.
\end{thm}

\medskip

Using the RIC-diagram and  Theorem \ref{Thm-EGS} repeatedly,  we have the structure theorem for minimal systems (Ellis-Glasner-Shapiro \cite{EGS},
Veech \cite{V77}, and Glasner \cite{G76}). We state the structure theorem in its relative form.

\begin{thm}\label{structure}
Given a factor map $\pi: X \to Y$ of minimal systems, there exists an ordinal $\eta$ (countable when $X$ is metrizable) and a canonically defined commutative diagram ({\em the canonical PI tower})
	\begin{equation*}
	\xymatrix
	{X \ar[d]_{\pi}             &
		X_0 \ar[l]_{{\tilde\theta}_0}
		\ar[d]_{\pi_0}
		\ar[dr]^{\sigma_1}         & &
		X_1 \ar[ll]_{{\tilde\theta}_1}
		\ar[d]_{\pi_1}
		\ar@{}[r]|{\cdots}         &
		X_{\nu}
		\ar[d]_{\pi_{\nu}}
		\ar[dr]^{\sigma_{\nu+1}}       & &
		X_{\nu+1}
		\ar[d]_{\pi_{\nu+1}}
		\ar[ll]_{{\tilde\theta}_{\nu+1}}
		\ar@{}[r]|{\cdots}         &
		X_{\eta}=X_{\infty}
		\ar[d]_{\pi_{\infty}}          \\
		Y                 &
		Y_0 \ar[l]^{\theta_0}          &
		Z_1 \ar[l]^{\rho_1}            &
		Y_1 \ar[l]^{\theta_1}
		\ar@{}[r]|{\cdots}         &
		Y_{\nu}                &
		Z_{\nu+1}
		\ar[l]^{\rho_{\nu+1}}          &
		Y_{\nu+1}
		\ar[l]^{\theta_{\nu+1}}
		\ar@{}[r]|{\cdots}         &
		Y_{\eta}=Y_{\infty}
	}
	\end{equation*}
where for each $\nu\le\eta, \pi_{\nu}$ is RIC, $\rho_{\nu}$ is equicontinuous, $\theta_{\nu}, {\tilde\theta}_{\nu}$ are proximal and 	$\pi_{\infty}$ is RIC and weakly mixing. For a limit ordinal $\nu ,\  X_{\nu}, Y_{\nu}, \pi_{\nu}$ etc. are the inverse limits of $ X_{\iota}, Y_{\iota}, \pi_{\iota}$ etc. for	$\iota < \nu$. Thus $X_\infty$ is a proximal extension of $X$ and a RIC weakly mixing extension of  $Y_\infty$.

The factor map $\pi_\infty$ is an isomorphism (so that $X_\infty=Y_\infty$) if and only if  $X$ is a PI extension of $Y$.

\end{thm}

\begin{rem}\label{rem-3.6}
\begin{enumerate}
  \item For any ordinal number $\alpha$, we can also define a $\tau$-closed normal subgroup  $H_{\alpha}(F)$ of $F$. Let $H_1(F)=H(F)$. Let $H_{\alpha+1}(F)=H(H_\alpha(F))$ if $\alpha$ is a successor ordinal; let $H_\alpha(F)=\cap_{\beta<\alpha}H_\beta(F)$ if $\alpha$ is a limit ordinal.

We omit the fixed points $x_0,y_0$ etc. and denote the Ellis groups of $X$ and $Y$ by $A$ and $F$ respectively, i.e. $\G(X)=A, \G(Y)=F$. Then for every $\nu$, the Ellis groups of $X_\nu$, $Y_\nu$ and $Z_\nu$ are as follows:
$$\G(X_\nu)=A, \quad \G(Y_\nu)=H_\nu(F)A, \quad \G(Z_\nu)=H_\nu (F)A.$$
  \item By Theorem \ref{Thm-EGS}, for $n\in \N$, $\pi_n$ is an isomorphism
		if and only if $H_n(F)\subseteq A$.
\end{enumerate}

\end{rem}

\subsection{PI tower of order $n$}\
\medskip

Let $\pi: (X,T)\rightarrow (Y,T)$ is a factor map of minimal systems.
For our purpose, we do not need the whole PI tower of $\pi$. Let $n\in \N$. We need the first $n$ step of  PI tower of $\pi$ in Theorem \ref{structure} as follows:
\begin{equation}\label{a1}
	\xymatrix{
	X\ar[d]_{\pi}&	{X}_0\ar[l]_{\tilde\theta_0}\ar[d]_{\pi_0}\ar[dr]^{\sigma_1} & & {X}_1\ar[d]_{\pi_1}\ar[ll]_{\tilde\theta_1}\ar@{}[r]|{\cdots} &{X}_{n-1}	\ar[d]_{\pi_{n-1}}\ar[dr]^{\sigma_{n}} & & {X}_{n}\ar[d]_{\pi_{n}}\ar[ll]_{\tilde\theta_{n}}	\\
	Y &{Y}_0\ar[l]^{\theta_0} & {Z}_1\ar[l]^{\rho_1} &{Y}_1\ar[l]^{\theta_1}	\ar@{}[r]|{\cdots} &{Y}_{n-1}  & {Z}_{n}\ar[l]^{\rho_{n}} & {Y}_{n}\ar[l]^{\theta_{n}}	
	}
	\end{equation}
We call the above diagram the {\em canonical tower of order $n$} associated with $\pi: X\rightarrow Y$. If in \eqref{a1} $\pi_n: X_n\rightarrow Y_n$ is an isomorphism, then we say that $n$ is the {\em order of $\pi$ }. Remark \ref{rem-3.6}-(2) gives a condition to decide the order of $\pi$. When $Y=\{pt\}$ is trivial, we also say that $n$ is the {\em order of $X$}.

\medskip

Note that each $\pi_m: X_m\rightarrow Y_m$ is RIC for $m<n$, but $\lambda_m: X_n\rightarrow Y_m$ is not RIC in general.
In \cite{G94}, it is shown that one can remedy this deficiency by enlarging the proximal extensions.

\begin{de}
If in (\ref{a1}), for $m<n$, $\lambda_m: X_n\rightarrow Y_m$ is RIC, then we call it is the {\em complete canonical tower of order $n$} associated with $\pi: X\rightarrow Y$.
\end{de}

\subsection{Largest distal factor of order $n$}\
\medskip

If $\pi: (X,T)\rightarrow (Y,T)$ is a distal extension of minimal t.d.s. Then in Theorem \ref{structure}, there are no proximal extensions and weakly mixing extension in the PI-tower. That is, PI tower of $\pi$ is as in Furstenberg Theorem:
\begin{equation}\label{a6}
  Y=Z_0 \stackrel{\rho_1} \longleftarrow  Z_1   \stackrel{\rho_2}\longleftarrow \cdots  \stackrel{\rho_n} \longleftarrow Z_n  \stackrel{\rho_{n+1}} \longleftarrow Z_{n+1}
 \longleftarrow \cdots \longleftarrow Z_\eta=X.
\end{equation}
Note that in \eqref{a6} for each $n<\eta$, the system $(Z_{n+1},T)$ is the largest equicontinuous extension of $Z_{n}$ within $X$. And in this case $\eta$ is called the {\em order of $\pi$}. For $n\le \eta$, we call $(Z_n,T)$ the {\em largest distal factor of order $n$} of $\pi$.

\medskip

When $Y=\{pt\}$ is trivial, $(X,T)$ is distal and $\eta$ is called the {\em order of $X$}. In this case for $n<\eta$, $(Z_n,T)$ is called the {\em largest distal factor of order $n$} of $X$. Note that for $n\le \eta$ $(Z_n,T)$ is also distal and its order is exactly $n$.

\begin{rem}
By Furstenberg structure theorem for distal systems, there is some ordinal $\eta$ corresponding to the number of successive equicontinuous extension needed to pass from the trivial system to the given one. Clearly this ordinal is not unique. But if we require that as in \eqref{a6}, for each $n<\eta$, the system $(Z_{n+1},T)$ is the largest equicontinuous extension of $Z_{n}$ within $X$ (this condition is called {\em normal} in \cite{F63}), then the tower is unique. For more discussion about order of a minimal distal system, refer to \cite[Section 13.]{F63}.
\end{rem}

\section{Decomposition of not totaly minimal systems}\label{section-decomposition}

Let $(X,T)$ be a minimal t.d.s. and $m\in \N$. In general, $(X,T^m)$ may not be minimal. But it has finitely many minimal components. In this section we collect some basic facts about the decomposition of $(X,T^m)$. The first one is easy to be verified (see \cite[Chapter II,\ 5.4]{vries} for example).

\begin{lem}\label{lem2}
	Let $(X,T)$ be a t.d.s. and $m \in \N$.  Then the $T$-minimal points of $(X,T)$ are identical with the $T^m$-minimal points of $(X,T^m)$.
\end{lem}

We need the following theorem about decomposition of $T^m$.

\begin{lem}\cite[Theorem 3.1]{ye}\label{lem3}
	Let $(X,T)$ be a minimal t.d.s. and $m \in \N$.  Then $X$ decomposes into a disjoint union of $l_m=l_m(X,T)\in\N$ sets
	\begin{equation*}
	X=X^{m,1}\cup \ldots \cup X^{m,l_m},
	\end{equation*}
	where $l_m$ divides $m$,  $TX^{m,j}= X^{m,j+1 \ ({\rm mod} \ l_m)}$, and the systems $(X^{m,j},T^m), j=1,\ldots,l_m$, are minimal.
\end{lem}

\begin{rem}\label{rem-decomposition}
\begin{enumerate}
  \item It is clear that for all $k\in \Z$, we have $T^kX^{m,j}= X^{m,j+k \ ({\rm mod} \ l_m)}$.
To save on notation, the second component of the index
on $X^{m, j+k}$ will be implicitly understood to be taken
modulo $l_m$. For example, we have that $X^{m,0}=X^{m,l_m}$, $X^{m,k}=X^{m,k+jl_m}$ for all $j, k\in \Z$ etc.
  \item Suppose $(X, T)$ is minimal but not totally minimal. Let $p_1$ be the smallest positive integer such that $(X,T^{p_1})$ is not minimal, then $p_1$ is prime and all minimal subsets of $(X,T^{p_1})$ are isomorphic. Let $X_1$ be such a minimal subset of $(X, T^{p_1})$, and let $T_1=T^{p_1}|_{X_1}$. If $(X_1, T_1)$ is not totally minimal, then the smallest positive integer $p_2$ such that $(X_1,T_1^{p_2})$ is not minimal is a prime with $p_2\ge p_1$. Continuing in this manner, we either obtain a totally minimal cascade $(X_n,T_n)$ or a sequence of primes $p_1\le p_2\le \ldots$. Refer to \cite[Page 31]{Au88} about this result.

\item Let $(X,T)$ be a minimal t.d.s. and $m ,n \in \N$ with $m|n$.  Then by Lemma \ref{lem3}, we have decompositions
	\begin{equation*}
	X=X^{m,1}\cup \ldots \cup X^{m,l_m},
	\end{equation*}
and \begin{equation*}
	X=X^{n,1}\cup \ldots \cup X^{n,l_n},
	\end{equation*}
	where $l_m | m$  and $l_n| n$ and the systems $(X^{m,j},T^m), j=1,\ldots,l_m$; $(X^{n,j},T^n), j=1,\ldots,l_n$ are minimal. Note that for each $k\in \{1,\ldots, l_n\}$, there is some $i\in \{1,2, \ldots, l_m\}$ such that $X^{n,k}\subseteq X^{m,i}$.

In fact, since $m| n$, there is some $k\in \N$ such that $n=km$. Now for each $i\in \{1,2, \ldots, l_m\}$, $(X^{m, i}, T^m)$ is minimal, and there is some $l_k=l_k(X^{m,i}, T^m)$ independent of $i$ such that
$$X^{m,i}=(X^{m,i})^{k,1}\cup\ldots \cup (X^{m,i})^{k,l_k},$$
where $((X^{m,i})^{k,j}, T^{mk}=T^n ), j=1,\ldots,l_k$ are minimal. It follows that for each $i\in \{1,2, \ldots, l_m\}, j\in \{1,2,\ldots, l_k\}$, there is a unique $k\in \{1,\ldots, l_n\}$ such that $(X^{m,i})^{k,j}=X^{n,k}$. Hence $l_n=l_ml_k$ and for each $k\in \{1,\ldots, l_n\}$, there is some $i\in \{1,2, \ldots, l_m\}$ such that $X^{n,k}\subseteq X^{m,i}$. 

\end{enumerate}

\end{rem}

\medskip

Let $(X,T)$ be a t.d.s., $A\subseteq X$ and $d\in \N$. Set
$$\Delta_d(A)=\{(x,x,\ldots,x): x\in A\}\subseteq X^d,$$
$$T^{(d)}=T\times \ldots\times T \ (d \ \text{times}),$$
$$\tau_d=\tau_d(T)=T\times T^2 \times \ldots \times T^{d}$$ and
$$\tau_d'=\tau'_d(T)=\id \times T\times \ldots \times T^{d-1}=\id\times \tau_{d-1}.$$
Note that $\langle\tau_d, T^{(d)}\rangle=\langle\tau'_d,T^{(d)}\rangle$, and $\D_d(X)$ is the diagonal of $X^d$, where $\langle\tau_d, T^{(d)}\rangle$ denotes the group generated by $\tau_d$ and $T^{(d)}$.

Let $(X,T)$ be a t.d.s. and $d\in N$. Let
$$N_d(X)=\overline{\O}(\D_d(X), \tau_d).$$
If $(X,T)$ is transitive and $x\in X$ is a transitive point, then
$$N_d({X})=\overline{\O}(x^{(d)},
\langle\tau_d, T^{(d)}\rangle),$$
the orbit closure of
$x^{(d)}=(x,\ldots,x)$ ($d$ times) under the action of the group
$\langle\tau_d, T^{(d)}\rangle$.

\medskip

The following result was given by Glasner in \cite{G94}, and one may find another proof in \cite[Chapter 1]{Glasner}.

\begin{thm}[Glasner]\label{thm2}
Let $(X,T)$ be a minimal t.d.s. and $d\in\N$. Then $(N_d(X), \langle\tau_d, T^{(d)}\rangle)$ is minimal and the $\tau_d$-minimal points are dense in $N_d(X)$.
\end{thm}

In fact, to get the fact of density of $\tau_d$-minimal points in $N_d(X)$, we may only assume $(X,T)$ has a dense set of minimal points.

\begin{lem}\label{lemma4.7}
Let $(X,T)$ be a t.d.s. and $d\in N$. If $X$ has a dense set of minimal points, then $N_d(X)$ has a dense subset of $\tau_d$-minimal points.
\end{lem}

\begin{proof}
Denote by $\mathcal{T}=\langle\tau_d, T^{(d)}\rangle$  the group generated by $\tau_d$ and $T^{(d)}$. Let $U\subseteq X^d$ be an open set intersect with $N_d(X)=\overline{\O}(\Delta_{d}(X),\tau_d)$. We can find some $x\in X$ such that
	\begin{equation*}
	\O(x^{d},\tau_d)\cap U\neq \emptyset.
	\end{equation*}
	Since the set of $T$-minimal points is dense in $X$, we can assume $x$ is a $T$-minimal point. By Theorem \ref{thm2},	$(\overline{\O}(x^{(d)},\mathcal{T}),\mathcal{T})$
	is minimal, and the set of $\tau_d$-minimal points is dense in $\overline{\O}(x^{(d)},\mathcal{T})$. Since $\O(x^{d},\tau_d)\cap U\neq \emptyset$, we have
$$\overline{\O}(x^{(d)},\mathcal{T}) \cap U\neq \emptyset.$$
As the set of $\tau_d$-minimal points is dense in $\overline{\O}(x^{(d)},\mathcal{T})$, one can find some $\tau_d$-minimal point in $U$. It follows that $N_d(X)$ has a dense subset of $\tau_d$-minimal points.
\end{proof}

The case $k=1$ of the following lemma is just Theorem 4.2 in \cite{GKR}.

\begin{lem}\label{lem6}
Let $k\in\N$ and $(X_i,T_i),1\leq i\leq k$ be minimal t.d.s. Let $m \in \N$ and  $M$ be the lowest common multiple of $1,2,\ldots,m$. Assume the decompositions of $X_i,1\leq i\leq k$  as in Lemma \ref{lem3} respect to $T_i^M$ are:
	\begin{equation*}
	X_i=X^{M,1}_i\cup\ldots\cup X^{M,l_M^i}_i,
	\end{equation*}
where $l_M^i|M$. Let $$(X,T)=(\prod_{i=1}^kX_i,\prod_{i=1}^kT_i)$$ and let $\tau_m=\tau_m(T)=T\times T^2 \times \ldots \times T^{m}: X^m\rightarrow X^m$.

Then $N_m(X)=\overline{\O}(\Delta_m(X),\tau_m)=\overline{\O}(\Delta_m(\prod_{i=1}^kX_i),\tau_m)$ has the following decomposition
\begin{equation*}
	N_m(X)=\bigcup_{1\le j_1 \le l^1_M ,\ldots, 1\le j_k\le l_M^k}
N_m(\prod_{i=1}^kX^{M,j_i}_i),
	\end{equation*}
where $N_m(\prod_{i=1}^k X^{M, j_i}_i)=\overline{\O}(\Delta_m(\prod_{i=1}^kX^{M, j_i}_i),\tau_m)$ are mutually disjoint.
In particular, $N_m(\prod_{i=1}^kX^{M, j_i}_i)$ is clopen in $N_m(X)$.
\end{lem}

\begin{proof}
Since $X=\prod_{i=1}^kX_i=\bigcup\limits_{1\le j_1 \le l^1_M ,\ldots,1\le j_k\le l_M^k}
\prod_{i=1}^kX^{M,j_i}_i$, it follows that
$$N_m(X)=\bigcup_{1\le j_1 \le l^1_M,\ldots, 1\le j_k\le l_M^k}
N_m(\prod_{i=1}^kX^{M,j_i}_i).$$
We need to show that $N_m(\prod_{i=1}^k X^{M, j_i}_i)$ are mutually disjoint.
Note that $\prod_{i=1}^kX^{M,j_i}_i$ are mutually disjoint and clopen in $X$, for every  $\overrightarrow{\bf x}\in N_m(\prod_{i=1}^k X^{M, j_i}_i)$, it is easy to deduce that  $\overrightarrow{\bf x}$ must be in a set of the following form:
\begin{equation*}
\prod_{i=1}^kX^{M, j_i+n}_i\times\prod_{i=1}^kX^{M, j_i+2n}_i \times\ldots\times\prod_{i=1}^kX^{M, j_i+mn}_i,
\end{equation*}
where $n\in\Z$.
Suppose that $1\le j_1,j'_1\le l_M^1 ,\ldots, 1\le j_k,j'_k\le l_M^k$ such that
$$N_m(\prod_{i=1}^kX^{M,j_i}_i)\cap N_m(\prod_{i=1}^kX^{M,j'_i}_i)\neq \emptyset.$$
It implies that there exist $n,n^\prime\in \Z$ such that
\begin{equation*}
\prod_{i=1}^kX^{M, j_i+n}_i\times\prod_{i=1}^kX^{M, j_i+2n}_i \times\ldots\times\prod_{i=1}^kX^{M, j_i+mn}_i
\end{equation*}
is equal to
\begin{equation*}
\prod_{i=1}^kX^{M, j_i^\prime+n^\prime}_i\times\prod_{i=1}^kX^{M, j_i^\prime+2n^\prime}_i \times\ldots\times\prod_{i=1}^kX^{M, j_i^\prime+mn^\prime}_i.
\end{equation*}
It follows that
$$j_1-j'_1\equiv s(n-n')  \ ({\rm mod}\ l_M^1),\text{ for all }1\leq s\leq m.$$
Thus $(s, l_M^1)| (j_1-j_1')$ for all $1\leq s\leq m$, and
$l_M^1=(M, l_M^1)|(j_1-j_1')$. It follows that $j_1 \equiv j_1'  \ ({\rm mod}\ l_M^1)$, that is, $j_1=j_1'$. Similarly we have  $j_1=j_1', \ldots ,j_k=j_k'$, which means that
$$N_m(\prod_{i=1}^kX^{M,j_i}_i)= N_m(\prod_{i=1}^kX^{M,j'_i}_i).$$
So $N_m(\prod_{i=1}^k X^{M, j_i}_i)$ are mutually disjoint, and the proof is completed.
\end{proof}

\section{Structure theorems of minimal but not totally minimal systems}\label{section-PI-tower}

In this section we will study the structure theorem of $(X,T^m)$. The main result of this section is that though for a minimal t.d.s. $(X,T)$ and $m\in \N$, $(X,T^m)$ may not be minimal, but we still can have PI-tower for $(X,T^m)$ and in fact it looks the same as the PI tower of $(X,T)$.

To get better understand of this point, first let us recall the measure version of this result.

\subsection{Furstenberg-Zimmer structure theorem}\
\medskip

In \cite{F77}, to prove multiple recurrence theorem, Furstenberg built a structure for ergodic systems. It says that each ergodic system is a weakly mixing extension of a distal system. For the general locally compact acting group actions, this was independently given by Zimmer \cite{Zimmer1, Zimmer2}.
To give the precise description of the theorem, we need some definitions.

Let $\pi: (X,\X,\mu,T)\rightarrow (Y,\Y,\nu,T)$ be a factor map of measure preserving systems
and $\mu=\int_Y \mu_y d \nu(y)$ the disintegration of $\mu$
relative to $\nu$. A function $f\in L^2(X,\X,\mu)$ is {\em almost periodic over $\Y$}
if for every $\ep>0$ there exist $g_1,\ldots, g_l \in
L^{2}(X,\X,\mu)$ such that for all $n\in \Z$,
$$\min_{1\le j\le l} ||T^nf-g_j||_{L^2(\mu_y)}<\ep$$
for $\nu$ almost every $y\in Y$. Let $K(X|Y, T)$ be the closed subspace of $L^2(X)$ spanned by the
almost periodic functions over $\Y$. When $\Y$ is trivial,
$K(X|Y,T)$ is $H_c=L^2(X,\mathcal{K}_\mu ,\mu)$, the closed subspace spanned by eigenfunctions of $T$.

We say that
$X$ is an {\em isometric extension} of $Y$ if $K(X|Y,T)=L^2(X)$ and
it is a {\em (measurable) weak mixing extension} of $Y$ if
$K(X|Y,T)=L^2(Y)$.

\begin{de}
Let $(X,\X,\mu,T)$ be an ergodic measure preserving system. Then $(X,\X,\mu,T)$ is
{\em measurable distal} if there exists a countable ordinal $\eta$
and a directed family of factors $(X_\theta, \X_\theta, \mu_\theta, T), \theta \le \eta$ such that
\begin{enumerate}
  \item $X_0=\{pt\}$ is the trivial system and $X_\eta=X$.
  \item For $\theta<\eta$ the extension $\pi_\theta : X_{\theta+1}\rightarrow X_\theta$ is isometric and nontrivial. 
  \item For a limit ordinal $\lambda \le \eta$, $X_\lambda={\lim\limits_{\longleftarrow}}_{\theta<\lambda} X_\theta$ (i.e. $\X_\lambda =\bigvee \X_\theta$).
\end{enumerate}
$$\{pt\}= X_0 \stackrel{\pi_0} \longleftarrow  X_1\stackrel{\pi_1} \longleftarrow \cdots \stackrel{\pi_{n-1}} \longleftarrow X_n \stackrel{\pi_n} \longleftarrow X_{n+1} \stackrel{\pi_{n+1}} \longleftarrow \cdots \longleftarrow X_{\eta} =X.$$
\end{de}

\begin{thm}[Furstenberg-Zimmer's structure theorem]\label{thm-FZ}
Each ergodic measure preserving system  $(X,\X,\mu,T)$ is a weakly mixing extension of
an ergodic distal system $X_\eta$.
$$\underbrace{(X_0,T) \stackrel{\pi_0} \longleftarrow  (X_1,T)\stackrel{\pi_1} \longleftarrow \cdots
\stackrel{\pi_{n-1}} \longleftarrow (X_n,T) \stackrel{\pi_n} \longleftarrow (X_{n+1},T)
\stackrel{\pi_{n+1}} \longleftarrow \cdots \longleftarrow
(X_{\eta},T) }_{\rm measurable \ distal\ factor} \stackrel{\pi_\eta}\longleftarrow (X,T).$$
\end{thm}

In \cite{F77}, to deal with $\tau_d=T\times T^2\times \ldots \times T^d$ for all $d\in \N$, one needs to study $T^m$, $1\le m\le d$. For ergodic system $(X,\X,\mu,T)$ and $m\in \N$, in general $(X,\X,\mu, T^m)$ is not necessarily ergodic. But in \cite{F77}, it is shown that $K(X|Y, T^m)=K(X|Y, T)$ for all $m$. Thus one can come to the following conclusion:

\begin{thm}\label{FZ-tower}
Let $(X,\X,\mu,T)$ be an ergodic measure preserving system  and $m\in \N$. Then $(X,\X,\mu,T^m)$ has the same structure theorem as $(X,\X,\mu,T)$ in Theorem \ref{thm-FZ}. That is, $(X,\X,\mu,T^m)$ is a measurable weakly mixing extension of
a measurable distal system $(X_\eta,\X_\eta, \mu, T^m)$:
$$(X_0,T^m) \stackrel{\pi_0} \longleftarrow  (X_1,T^m)\stackrel{\pi_1} \longleftarrow \cdots
\stackrel{\pi_{n-1}} \longleftarrow (X_n,T^m) \stackrel{\pi_n} \longleftarrow (X_{n+1},T^m)
\stackrel{\pi_{n+1}} \longleftarrow \cdots \longleftarrow
(X_{\eta},T^m)  \stackrel{\pi_\eta}\longleftarrow (X,T^m).$$
\end{thm}

In this section we will give the topological version of Theorem \ref{FZ-tower}.

\subsection{PI-towers for $T^m$}\
\medskip

In this subsection we will give PI-tower for $T^m$. Let $\pi:(X,T)\rightarrow (Y,T)$ be a factor map of minimal t.d.s. and $m\in \N$. Let the PI-tower of $\pi:(X,T)\rightarrow (Y,T)$ be:
	\begin{equation*}
	\xymatrix{
		X\ar[d]_{\pi}&	X_0\ar[l]_{\tilde\theta_0}\ar[d]_{\pi_0}\ar[dr]^{\sigma_1} & & X_1\ar[d]_{\pi_1}\ar[ll]_{\tilde\theta_1}\ar@{}[r]|{\cdots} &X_v	\ar[d]_{\pi_v}\ar[dr]^{\sigma_{v+1}} & & X_{v+1}\ar[d]_{\pi_{v+1}}\ar[ll]_{\tilde\theta_{v+1}}	\ar@{}[r]|{\cdots} &X_\eta=X_\infty\ar[d]_{\pi_\infty} \\
		Y&Y_0\ar[l]^{\theta_0}  & Z_1\ar[l]^{\rho_1} & Y_1\ar[l]^{\theta_1}	\ar@{}[r]|{\cdots} &Y_{v}  & Z_{v+1}\ar[l]^{\rho_{v+1}} & Y_{v+1}\ar[l]^{\theta_{v+1}}	\ar@{}[r]|{\cdots} &Y_\eta=Y_\infty
	}
	\end{equation*}
In general $(X,T^m)$ is not minimal, but we will show that $\pi: (X,T^m)\rightarrow (Y,T^m)$ has the same PI-tower as $\pi:(X,T)\rightarrow (Y,T)$. In fact we will give the all PI-towers of minimal components of $(X,T^m)$ (see Theorem \ref{PI-not-totaly-min}). First we need some lemmas.

\begin{lem}\label{lem1}
Let $\pi:(X,T)\rightarrow (Y,T)$ be a factor map of t.d.s., and $m \in \N$. Then $Q_\pi(X,T)=Q_\pi(X,T^m)$ and $P_\pi(X,T)=P_\pi(X,T^m)$.		
\end{lem}
\begin{proof}
Let $m\in \N$.
It is clear that $Q_\pi(X,T^m)\subseteq Q_\pi(X,T)$. Now we show that $Q_\pi(X,T)\subseteq Q_\pi(X,T^m)$.
	Let $(x,z)\in Q_\pi(X,T)$. Then by definition  there exist $\{(x_i,z_i)\}_{i\in\N}\subseteq R_\pi$ and $\{n_i\}_{i\in\N}\subseteq\Z$ such that
	\begin{equation*}
	(x_i,z_i)\rightarrow (x,z), \ \rho(T^{n_i}x_i,T^{n_i}z_i) \rightarrow 0, i\to\infty.
	\end{equation*}
	Passing to a subsequence we may assume there exists $0\leq m_0\leq m-1$ such that $n_i\equiv m_0 \ (\mod \ m)$ for all $i\in\N$.
	Let $n_i=mk_i+m_0$, $k_i\in\Z$. Then
	\begin{equation*}
	\rho (T^{mk_i+m_0}x_i,T^{mk_i+m_0}z_i)\rightarrow 0, i\to\infty.
	\end{equation*}
	Note that $T^{m-m_0}$ is uniformly continuous, we have
	\begin{equation*}
	\rho (T^{m(k_i+1)}x_i,T^{m(k_i+1)}z_i)\rightarrow 0, i\to\infty.
	\end{equation*}
	It follows that $(x,z)\in Q_\pi(X,T^m)$. That is, $Q_\pi(X,T)\subseteq Q_\pi(X,T^m)$.

\medskip

Similarly we can show that $P_\pi(X,T)=P_\pi(X,T^m)$ for all $m\in \N$.
\end{proof}	

\begin{lem}\label{lem-prox}
Let $\pi: (X,T)\rightarrow (Y,T)$ be a factor map of t.d.s. If $\pi$ is proximal, then there is one to one correspondence between the set of minimal subsystems of $(X,T)$ and $(Y,T)$.
\end{lem}

\begin{proof}
It suffices to show that if there are minimal subsets $X_1,X_2$ of $X$ with $\pi(X_1)=\pi(X_2)$, then $X_1=X_2$. Since $\pi(X_1)=\pi(X_2)$, there are some $x_1\in X_1$ and $x_2\in X_2$ such that $\pi(x_1)=\pi(x_2)$. Since $\pi$ is proximal, $x_1,x_2$ are proximal. Hence there are some sequence $\{n_i\}_{i=1}^\infty\subseteq \Z$  and $z\in X$ such that
$$(T^{n_i}x_1, T^{n_i}x_2)\rightarrow (z,z), \ i\to \infty.$$
Since $x_1\in X_1$ and $X_1$ is a $T$-invariant closed set, $T^{n_i}x_1\in X_1$ for all $i$ and hence $z\in X_1$. By the same reason, $z\in X_2$. It follows that $z\in X_1\cap X_2$. Thus $X_1=X_2$ by the minimality of $X_1$ and $X_2$.
\end{proof}

\begin{thm}\label{thm2.4}\label{PI-not-totaly-min}
Let $\pi:(X,T)\rightarrow (Y,T)$ be a factor map of minimal t.d.s., and $m\in \N$. Let the PI-tower of $\pi:(X,T)\rightarrow (Y,T)$ be:
	\begin{equation*}
	\xymatrix{
		X\ar[d]_{\pi}&	X_0\ar[l]_{\tilde\theta_0}\ar[d]_{\pi_0}\ar[dr]^{\sigma_1} & & X_1\ar[d]_{\pi_1}\ar[ll]_{\tilde\theta_1}\ar@{}[r]|{\cdots} &X_v	\ar[d]_{\pi_v}\ar[dr]^{\sigma_{v+1}} & & X_{v+1}\ar[d]_{\pi_{v+1}}\ar[ll]_{\tilde\theta_{v+1}}	\ar@{}[r]|{\cdots} &X_\eta=X_\infty\ar[d]_{\pi_\infty} \\
		Y&Y_0\ar[l]^{\theta_0}  & Z_1\ar[l]^{\rho_1} & Y_1\ar[l]^{\theta_1}	\ar@{}[r]|{\cdots} &Y_{v}  & Z_{v+1}\ar[l]^{\rho_{v+1}} & Y_{v+1}\ar[l]^{\theta_{v+1}}	\ar@{}[r]|{\cdots} &Y_\eta=Y_\infty
	}
	\end{equation*}
By Lemma \ref{lem3}, $X$ decomposes into a disjoint union of $l_m=l_m(X,T)\in\N$ sets
	\begin{equation*}
	X=X^{m,1}\cup \ldots \cup X^{m,l_m},
	\end{equation*}
	where $l_m |m$,  $TX^{m,j}= X^{m,j+1 \ ({\rm mod} \ l_m)}$, and the systems $(X^{m,j},T^m), j=1,\ldots,l_m$ are minimal.

Then for each $\nu\le \eta$, we also have the decompositions:
\begin{equation*}
	X_\nu=X_\nu^{m,1}\cup \ldots \cup X_\nu^{m,l_m}, \nu\ge 0
\end{equation*}
\begin{equation*}
	Y_\nu=Y_\nu^{m,1}\cup \ldots \cup Y_\nu^{m,l_m}, \nu\ge 1
\end{equation*}
\begin{equation*}
	Z_\nu=Z_\nu^{m,1}\cup \ldots \cup Z_\nu^{m,l_m}, \nu\ge 1
\end{equation*}
such that
\begin{enumerate}
  \item  $TH_\nu^{m,j}= H_\nu^{m,j+1 \ ({\rm mod} \ l_m)}$, and the systems $(H_\nu^{m,j},T^m), j=1,\ldots,l_m$ are minimal, where $H$ is one of $X,Y,Z$.
  \item $\pi_\nu(X_\nu^{m,j})=Y_\nu^{m,j};\sigma_{\nu+1}(X_\nu^{m,j})=Z_{\nu+1}^{m,j};\rho_{\nu+1}
      ({Z_{\nu+1}^{m,j}})=Y_\nu^{m,j},\ 1\leq j\leq l_m$.
  \item the PI-towers of $\pi:(X^{m,j},T^m)\rightarrow(\pi(X^{m,j}),T^m),1\leq j\leq l_m$ are:
\end{enumerate}
	\begin{equation*} \footnotesize
	\xymatrix{
	X^{m,j}\ar[d]_{\pi}&	X_0^{m,j}\ar[l]_{\tilde\theta_0}\ar[d]_{\pi_0}\ar[dr]^{\sigma_1} & & X_1^{m,j}\ar[d]_{\pi_1}\ar[ll]_{\tilde\theta_1}\ar@{}[r]|{\cdots} &X_\nu^{m,j}	\ar[d]_{\pi_\nu}\ar[dr]^{\sigma_{\nu+1}} & & X_{\nu+1}^{m,j}\ar[d]_{\pi_{\nu+1}}\ar[ll]_{\tilde\theta_{\nu+1}}	\ar@{}[r]|{\cdots} &X_\eta^{m,j}=X_\infty^{m,j}\ar[d]_{\pi_\infty} \\
	\pi(X^{m,j})&\pi_{0}(X_0^{m,j})\ar[l]^{\theta_0} & Z_1^{m,j}\ar[l]^{\rho_1} & Y_1^{m,j}\ar[l]^{\theta_1}	\ar@{}[r]|{\cdots} &Y_{\nu}^{m,j}  & Z_{\nu+1}^{m,j}\ar[l]^{\rho_{\nu+1}} & Y_{\nu+1}^{m,j}\ar[l]^{\theta_{\nu+1}}	\ar@{}[r]|{\cdots} &Y_\eta^{m,j}=Y_\infty^{m,j}
	}
	\end{equation*}
\end{thm}

\begin{proof}
Let $m\in\N$ be fixed. By Lemma \ref{lem3}, $X$ decomposes into a disjoint union of $l_m=l_m(X,T)\in\N$ sets
	\begin{equation*}
	X=X^{m,1}\cup \ldots \cup X^{m,l_m},
	\end{equation*}
	where $l_m |m$,  $TX^{m,j}= X^{m,j+1 \ ({\rm mod} \ l_m)}$, and the systems $(X^{m,j},T^m), j=1,\ldots,l_m$ are minimal.

By Lemma \ref{lem1},  for each $\nu\le \eta$, $P_{\tilde{\theta_\nu}}(X_{\nu},T)=P_{\tilde{\theta_\nu}}(X_{\nu},T^m)$. So $\tilde\theta_{v}: (X_{\nu}, T^m)\rightarrow (X_{\nu-1},T^m)$ is proximal.
By Lemma \ref{lem-prox}, for each $\nu\le \eta$,
\begin{equation}\label{s1}
	X_\nu=X_\nu^{m,1}\cup \ldots \cup X_\nu^{m,l_m},
	\end{equation}
	where $TX_\nu^{m,j}= X_\nu^{m,j+1 \ ({\rm mod} \ l_m)}$, and the systems $(X_\nu^{m,j},T^m), j=1,\ldots,l_m$ are minimal.

Now we show that for all $\nu\le \eta$, $Z_\nu$ has just $l_m$ $T^m$-minimal subsets. Recall in the following diagram, $Z_{\nu+1}=X_\nu/ Q_{\pi_{\nu}}(X_\nu,T)$, i.e. $R_{\sigma_{\nu+1}}=Q_{\pi_{\nu}}(X_\nu,T)$.
	\begin{equation*}
	\xymatrix{
		X_\nu\ar[d]_{\pi_\nu}\ar[dr]^{\sigma_{\nu+1}} \\
		Y_\nu  & Z_{\nu+1}\ar[l]^{\rho_{\nu+1}}	.
	}
	\end{equation*}
Let $$\pi_\nu(X_\nu^{m,j})=Y_\nu^{m,j};\ \ \sigma_{\nu+1}(X_\nu^{m,j})=Z_{\nu+1}^{m,j},\ 1\leq j\leq l_m.$$
We need to show that $Z_{\nu+1}^{m,j}, j=1,\ldots,l_m$ are disjoint. If not, without loss of generality assume $\sigma_{\nu+1}(X_\nu^{m,1})=Z_{\nu+1}^{m,1}=Z_{\nu+1}^{m,h}=\sigma_{\nu+1}(X_\nu^{m,h})$ for some $h\neq 1$.
	Choose
	\begin{equation*}
	(x_1,x_h)\in (X_\nu^{m,1}\times X_\nu^{m,h})\cap R_{\sigma_{\nu+1}}.
	\end{equation*}
By Lemma \ref{lem1}
	\begin{equation*}
	R_{\sigma_{\nu+1}}= Q_{\pi_\nu}(X_\nu,T)=Q_{\pi_\nu}(X_\nu,T^m),
	\end{equation*}
and there are $\{z_1^i\}\subseteq X_\nu, \{z_h^i\}\subseteq X_\nu$ and $\{n_i\}\subseteq \Z$ such that
	\begin{equation*}
	(z_1^i,z_h^i)\rightarrow (x_1,x_h), (T^{mn_i}z_1^i,T^{mn_i}z_h^i)\rightarrow(z,z), \ i\to\infty.
	\end{equation*}
Since $(x_1,x_h)\in X_\nu^{m,1}\times X_\nu^{m,h}$ and $X_\nu^{m,1}\times X_\nu^{m,h}$ is open,  $(z_1^i,z_h^i)\in  X_\nu^{m,1}\times X_\nu^{m,h}$ for sufficiently large $i\in\N$. It follows that
	\begin{equation*}
	(T^{mn_i}z_1^i,T^{mn_i}z_h^i)\in  X_\nu^{m,1}\times X_\nu^{m,h}
	\end{equation*}
	for sufficiently large $i\in\N$.
	Hence $(z,z)\in X_\nu^{m,1}\times X_\nu^{m,h}$, i.e. $X_\nu^{m,1}\cap X_\nu^{m,h}\neq \emptyset$. A contradiction! So $Z_{\nu+1}^{m,j}, j=1,\ldots,l_m$ are disjoint.
Thus for each $\nu\le \eta$ with $\nu\ge 1$, $Z_{\nu}$ has the following decomposition:
\begin{equation}\label{s2}
	Z_{\nu}=Z_\nu^{m,1}\cup \ldots \cup Z_\nu^{m,l_m},
\end{equation}
where $TZ_\nu^{m,j}= Z_\nu^{m,j+1 \ ({\rm mod}\ l_m)}$, and the systems $(Z_\nu^{m,j},T^m), j=1,\ldots,l_m$ are minimal.

Note that for each $\nu\le \eta$ with $\nu\ge 1$ we have the following extensions:
$$(Z_\nu,T^m) \stackrel{\theta_\nu} \longleftarrow  (Y_\nu,T^m)\stackrel{\pi_\nu} \longleftarrow (X_\nu,T^m) .$$
By (\ref{s1}) and (\ref{s2}), we have for each $\nu\le \eta$ with $\nu\ge 1$, $Y_{\nu}$ has the following decomposition:
\begin{equation}\label{s3}
	Y_{\nu}=Y_\nu^{m,1}\cup \ldots \cup Y_\nu^{m,l_m},
\end{equation}
where $T Y_\nu^{m,j}= Y_\nu^{m,j+1 \ ({\rm mod} \ l_m)}$, and the systems $(Y_\nu^{m,j},T^m), j=1,\ldots,l_m$ are minimal.

\medskip

It is left to show that the PI-towers of $\pi:(X^{m,j},T^m)\rightarrow(\pi(X^{m,j}),T^m),1\leq j\leq l_m$ are:
	\begin{equation*} \footnotesize
	\xymatrix{
	X^{m,j}\ar[d]_{\pi}&	X_0^{m,j}\ar[l]_{\tilde\theta_0}\ar[d]_{\pi_0}\ar[dr]^{\sigma_1} & & X_1^{m,j}\ar[d]_{\pi_1}\ar[ll]_{\tilde\theta_1}\ar@{}[r]|{\cdots} &X_\nu^{m,j}	\ar[d]_{\pi_\nu}\ar[dr]^{\sigma_{\nu+1}} & & X_{\nu+1}^{m,j}\ar[d]_{\pi_{\nu+1}}\ar[ll]_{\tilde\theta_{\nu+1}}	\ar@{}[r]|{\cdots} &X_\eta^{m,j}=X_\infty^{m,j}\ar[d]_{\pi_\infty} \\
	\pi(X^{m,j})&\pi_{0}(X_0^{m,j})\ar[l]^{\theta_0} & Z_1^{m,j}\ar[l]^{\rho_1} & Y_1^{m,j}\ar[l]^{\theta_1}	\ar@{}[r]|{\cdots} &Y_{\nu}^{m,j}  & Z_{\nu+1}^{m,j}\ar[l]^{\rho_{\nu+1}} & Y_{\nu+1}^{m,j}\ar[l]^{\theta_{\nu+1}}	\ar@{}[r]|{\cdots} &Y_\eta^{m,j}=Y_\infty^{m,j}.
	}
\end{equation*}
	
Fix $j\leq l_m$. By Lemma \ref{lem1} for each $\nu\le \eta$, $\tilde\theta_{\nu}$ and $\theta_{\nu}$ are proximal with respect to the action $T^m$.

Now for each $\nu\le \eta$, we show that $\pi_\nu:(X_\nu^{m,j},T^m)\rightarrow (Y_\nu^{m,j},T^m)$ is $RIC$. Since $\pi_\nu:(X_\nu,T)\rightarrow (Y_\nu,T)$ is $RIC$ , $\pi_\nu$ is open and $R_{\pi_\nu}^n$ has dense $T$-minimal points for every $n\in\N$. By Lemma \ref{lem2}, $R_{\pi_\nu}^n$ has dense $T^m$-minimal points for every $n\in\N$. Note that $X_\nu^{m,j}$ is open, and it follows that  $\pi_\nu:(X_\nu^{m,j},T^m)\rightarrow (Y_\nu^{m,j},T^m)$ is $RIC$.

We next show that $\rho_{\nu+1}$ is the largest equicontinuous factor of $\pi_\nu$ with respect to $T^m$. To this aim, it needs to show that
	\begin{equation*}
	R_{\sigma_{\nu+1}}\cap (X_\nu^{m,j}\times X_\nu^{m,j})=Q_{\pi_\nu}(X_\nu^{m,j},T^m).
	\end{equation*}
	Note that
	\begin{equation*}
	R_{\sigma_{\nu+1}}=Q_{\pi_\nu}(X_\nu,T)=Q_{\pi_\nu}(X_\nu,T^m).
	\end{equation*}
By the  definition of $Q_{\pi_\nu}(X_\nu,T^m)$ we have
	\begin{equation*}
	\begin{split}
	R_{\sigma_{\nu+1}}\cap & (X_\nu^{m,j}\times X_\nu^{m,j})=
	\{(x_1,x_2)\in X_\nu^{m,j}\times X_\nu^{m,j}:
	\exists \ (z_1^i,z_2^i)\in R_{\pi_\nu},\\
	& n_i\in \Z \ s.t.\ (z_1^i,z_2^i)\rightarrow (x_1,x_2), \rho (T^{mn_i}z_1^i,T^{mn_i}z_2^i)\rightarrow 0\}.
	\end{split}
	\end{equation*}
	Note that $X_\nu^{m,j}\times X_\nu^{m,j}$ is open, the set on the right side is just $Q_{\pi_\nu}(X_\nu^{m,j},T^m)$.

The last thing left is to show that $\pi_\infty:(X_\infty^{m,j},T^m)\rightarrow (Y_\infty^{m,j},T^m)$  is weakly mixing. Since
	\begin{equation*}
	R_{\pi_\infty}=Q_{\pi_\infty}(X_\infty,T)=Q_{\pi_\infty}(X_\infty,T^m),
	\end{equation*}
	it follows that
	\begin{equation*}
	Q_{\pi_\infty}(X_\infty^{m,j},T^m)=Q_{\pi_\infty}(X_\infty,T^m)\cap X_\infty^{m,j}\times X_\infty^{m,j}=R_{\pi_\infty}\cap X_\infty^{m,j}\times X_\infty^{m,j}.
	\end{equation*}
	That means $\pi_\infty:(X_\infty^{m,j},T^m)\rightarrow (Y_\infty^{m,j},T^m)$ has trivial maximal equicontinuous factor. Since it is RIC, it is weakly mixing. The proof is completed.
\end{proof}

\begin{rem}
\begin{enumerate}
  \item In Theorem \ref{PI-not-totaly-min}, for $T^m$, $Y$ and $Y_0$ have the decompositions
\begin{equation*}
	Y=Y^{m,1}\cup \ldots \cup Y^{m,l'_m},
	\end{equation*}
and
\begin{equation*}
	Y_0=Y_0^{m,1}\cup \ldots \cup Y_0^{m,l'_m},
	\end{equation*}
where $l'_m=l_m(Y,T)=l_m(Y_0,T)$ ($l_m(Y,T)=l_m(Y_0,T)$ follows from the fact that $\theta_0$ is proximal). It is obvious that $l'_m\le l_m$, and it may happen that $l'_m\neq l_m$. For example $(X,T^m)$ is not minimal and $(Y,T)$ is trivial. In this case $l'_m=1< l_m$.
  \medskip
  \item If $(X,T)$ is a totally minimal system, then $l_m=1$ for every $m \in \N$.
\end{enumerate}

\end{rem}


\section{A generalization of Glasner's theorem}\label{section-Main}

In this section we will give a generalization of Glasner's theorem in \cite{G94}.

\subsection{Statement of the results}\
\medskip

In this section, we deal with topological characteristic factors along arithmetic progressions.
The main theorem  of this section is:

\begin{thm}\label{main}
	Let $\pi^i:(X_i,T_i)\rightarrow (Y_i,T_i)$ be factor maps of minimal t.d.s., $ 1\leq i\leq k$, where $k\in \N$. Let $n\in\N$ and assume the complete canonical PI-towers of order $n$  of $\pi^i:(X_i,T_i)\rightarrow (Y_i,T_i),1\leq i\leq k$ are as follows:
	\begin{equation*}
	\xymatrix{
	X_i\ar[d]_{\pi^i}&	(X_i)_0\ar[l]_{\tilde\theta^i_0}\ar[d]_{\pi^i_0}\ar[dr]^{\sigma^i_1} & & (X_i)_1\ar[d]_{\pi^i_1}\ar[ll]_{\tilde\theta_1^i}\ar@{}[r]|{\cdots} & (X_i)_{n-1}	\ar[d]_{\pi^i_{n-1}}\ar[dr]^{\sigma^i_{n}} & & (X_i)_{n}\ar[d]_{\pi_{n}^i}\ar[ll]_{\tilde\theta_{n}^i}	\\
	Y_i & (Y_i)_0\ar[l]^{\theta_0^i} & (Z_i)_1\ar[l]^{\rho^i_1} & (Y_i)_1\ar[l]^{\theta^i_1}	\ar@{}[r]|{\cdots} & (Y_i)_{n-1}  & (Z_i)_{n}\ar[l]^{\rho^i_{n}} & (Y_i)_{n}\ar[l]^{\theta^i_{n}},	
	}
	\end{equation*}
Then $(\prod\limits_{i=1}^k(Y_i)_{n},T)$ is an $n$-step topological characteristic factor of $(\prod\limits_{i=1}^k(X_i)_{n},T)$, where $T=T_1\times \ldots \times T_k$.

If in addition $(\prod\limits_{i=1}^k(X_i)_{n},T)$ is transitive, then $(\prod\limits_{i=1}^k(Y_i)_{n},T)$ is an $(n+1)$-step topological characteristic factor of $(\prod\limits_{i=1}^k(X_i)_{n},T)$.
\end{thm}

\begin{rem}
When $k=1$ and $Y_1$ is a trivial system, it is Glasner's result in \cite{G94}. We generalize Glasner's result in two ways. First it is a relative version; secondly we study the system $(X,T)$ with the form $(X_1\times \ldots \times X_k,T_1\times \ldots \times T_k)$, in which case $(X,T)$ is not minimal in general but it remains lots of good properties of minimal systems.
\end{rem}

\begin{cor}
Let $\pi: (X,T)\rightarrow (Y,T)$ be a factor map of minimal t.d.s. and let $k, n\in \N$. Assume the complete canonical PI-tower of order $n$  of $\pi$ is as follows:
\begin{equation*}
	\xymatrix{
	X\ar[d]_{\pi}&	{X}_0\ar[l]_{\tilde\theta_0}\ar[d]_{\pi_0}\ar[dr]^{\sigma_1} & & {X}_1\ar[d]_{\pi_1}\ar[ll]_{\tilde\theta_1}\ar@{}[r]|{\cdots} &{X}_{n-1}	\ar[d]_{\pi_{n-1}}\ar[dr]^{\sigma_{n}} & & {X}_{n}\ar[d]_{\pi_{n}}\ar[ll]_{\tilde\theta_{n}}	\\
	Y &{Y}_0\ar[l]^{\theta_0} & {Z}_1\ar[l]^{\rho_1} &{Y}_1\ar[l]^{\theta_1}	\ar@{}[r]|{\cdots} &{Y}_{n-1}  & {Z}_{n}\ar[l]^{\rho_{n}} & {Y}_{n}\ar[l]^{\theta_{n}}	.
	}
	\end{equation*}
Then $(Y_n^{k}, T^{(k)})$ is an $n$-step topological characteristic factor of $(X_n^k, T^{(k)})$, and $(Y_n, T)$ is an $(n+1)$-step topological characteristic factor of $(X_n, T)$.
\end{cor}

\begin{cor}\label{cor-main-distal}
Let $\pi: (X,T)\rightarrow (Y,T)$ be a distal extension of minimal t.d.s. and let $k, n\in \N$. Let $Y_n$ be the largest distal factor of $\pi$ of order $n$. Then $(Y_n^{k}, T^{(k)})$ is an $n$-step topological characteristic factor of $(X^k, T^{(k)})$, and $(Y_n, T)$ is an $(n+1)$-step topological characteristic factor of $(X, T)$.
\end{cor}

\begin{cor}\label{cor-main-weak}
Let $\pi: (X,T)\rightarrow (Y,T)$ be a RIC weakly mixing extension of minimal t.d.s. and let $k\in \N$. Then $(Y^{k}, T^{(k)})$ is an $n$-step topological characteristic factor of $(X^k, T^{(k)})$ for all $n\in \N$.
\end{cor}

\begin{cor}
Let $(X_1,T_1), \ldots, (X_k,T_k)$ be weakly mixing minimal t.d.s. with $k\in \N$. Then the trivial system is an $n$-step topological characteristic factor of $(X_1\times \ldots \times X_k,T_1\times \ldots \times T_k)$  for all $n\in \N$. In particular, $(X_1\times \ldots \times X_k,T_1\times \ldots \times T_k)$ is transitive.
\end{cor}

\subsection{Proof of Theorem \ref{main}}\
\medskip

To prove Theorem \ref{main}, we need the following lemma in \cite{G94}, which plays a key role in the proof.

\begin{lem}\cite[Lemma 3.3]{G94}\label{lem-Glasner}
	Let $k\in\N$ and $(X_i,T_i), (Y_i,T_i), 1\le i\le k$ be minimal t.d.s. Fix some $u\in J(M)$. For each $1\leq i\leq k$, choose $x^0_i\in X_i$ with $ux^0_i=x^0_i$ and let $\pi_i(x^0_i)=y^0_i,\sigma_i(x^0_i)=z^0_i$. Let $A_i=\G(X_i,x^0_i),F_i=\G(Y_i,y^0_i),1\leq i\leq k$.

Suppose that for each $1\leq i\leq k$, a commutative diagram
	\begin{equation*}
	\xymatrix{
		(X_i,T_i)\ar[d]_{\pi_i} \ar[dr]^{\sigma_i}      \\
		(Y_i,T_i)  & (Z_i,T_i)\ar[l]^{\rho_i}
	}
	\end{equation*}
is given, in which $B_i=\G(Z_i,z^0_i)=H(F_i)A_i,1\leq i\leq k$.
Let
	\begin{equation*}
	(X,T)=(\prod_{i=1}^k X_i,T),\ (Y,T)=(\prod_{i=1}^k Y_i,T),\ (Z,T)=(\prod_{i=1}^k Z_i,T),
	\end{equation*}
where $T=T_1\times\ldots \times T_k $, and let $\pi=\pi_1\times \ldots \times \pi_k,\sigma=\sigma_1\times \ldots \times\sigma_k, \rho=\rho_1\times \ldots \times \rho_k$ denote the product homomorphisms respectively.
	\begin{equation*}
	\xymatrix{
		(X,T)\ar[d]_{\pi} \ar[dr]^{\sigma}      \\
		(Y,T)  & (Z,T)\ar[l]^{\rho}
	}
	\end{equation*}
	Let $N_Y$ be a closed $T$-invariant subset of $Y$ and let $N_X=\pi^{-1}(N_Y),N_Z=\rho^{-1}(N_Y)$.

Let $Q$ be a closed subset of $N_X$. Suppose that
\begin{enumerate}
 \item $\sigma_i$ is open for each $1\leq i\leq k$ and the set of $T$-minimal points is dense in $N_X$,
   \item $\overline{\O}(Q, T)=N_X$,
   \item for every relatively open subset $U$ of $Q$, ${\rm int}_{N_X}(\overline{\O}(U, T))\neq\emptyset.$
\end{enumerate}
Then there exists a dense $G_\delta$ subset $\Omega$ of $Q$ such that ${\bf x}\in \Omega$ implies $\overline{\O}({\bf x})$ is $\sigma$-saturated.
\end{lem}

Now we prove Theorem \ref{main}:

\begin{proof}[Proof of Theorem \ref{main}]
We show the case $k=2$, and for general $k\in\N$, it is essentially the same. 	

To be precise, let $\pi^i:(X_i,T_i)\rightarrow (Y_i,T_i)$ be factor maps of minimal t.d.s. $(X_i,T_i)$ and $(Y_i, T_i)$, $1\leq i\leq 2$. Let $n\in\N$ and assume the complete canonical PI-towers of order $n$ of $\pi^i:(X_i,T_i)\rightarrow (Y_i,T_i),1\leq i\leq 2$ are:
	\begin{equation}\label{s8}
	\xymatrix{
	X_i\ar[d]_{\pi^i}&	(X_i)_0\ar[l]_{\tilde\theta^i_0}\ar[d]_{\pi^i_0}\ar[dr]^{\sigma^i_1} & & (X_i)_1\ar[d]_{\pi^i_1}\ar[ll]_{\tilde\theta_1^i}\ar@{}[r]|{\cdots} & (X_i)_{n-1}	\ar[d]_{\pi^i_{n-1}}\ar[dr]^{\sigma^i_{n}} & & (X_i)_{n}\ar[d]_{\pi_{n}^i}\ar[ll]_{\tilde\theta_{n}^i}	\\
	Y_i & (Y_i)_0\ar[l]^{\theta_0^i} & (Z_i)_1\ar[l]^{\rho^i_1} & (Y_i)_1\ar[l]^{\theta^i_1}	\ar@{}[r]|{\cdots} & (Y_i)_{n-1}  & (Z_i)_{n}\ar[l]^{\rho^i_{n}} & (Y_i)_{n}\ar[l]^{\theta^i_{n}},	
	}
	\end{equation}
in which $\lambda^i_m: (X_i)_n\rightarrow (Y_i)_m$ is RIC for each $m\in \{0,1,\ldots,n\}$.
Let $$X=X_1\times X_2, \quad Y=Y_1\times Y_2,$$ and for all $m\in \{0,1,\ldots, n\}$,
$$\widetilde{X}_m=(X_1)_m\times (X_2)_m,\ \widetilde{Y}_m=(Y_1)_m\times (Y_2)_m, \ \widetilde{Z}_m=(Z_1)_m\times (Z_2)_m,$$
$$\pi_m=\pi^1_m\times \pi^2_m, \sigma_m=\sigma^1_m\times \sigma^2_m, \rho_m=\rho^1_m\times \rho^2_m, \theta_m=\theta^1_m\times \theta^2_m, \  etc. $$
Hence we have the diagram:
\begin{equation*}
	\xymatrix{
	X\ar[d]_{\pi}&	\widetilde{X}_0\ar[l]_{\tilde\theta_0}\ar[d]_{\pi_0}\ar[dr]^{\sigma_1} & & \widetilde{X}_1\ar[d]_{\pi_1}\ar[ll]_{\tilde\theta_1}\ar@{}[r]|{\cdots} &\widetilde{X}_{n-1}	\ar[d]_{\pi_{n-1}}\ar[dr]^{\sigma_{n}} & & \widetilde{X}_{n}\ar[d]_{\pi_{n}}\ar[ll]_{\tilde\theta_{n}}	\\
	Y &\widetilde{Y}_0\ar[l]^{\theta_0} & \widetilde{Z}_1\ar[l]^{\rho_1} & \widetilde{Y}_1\ar[l]^{\theta_1}	\ar@{}[r]|{\cdots} &\widetilde{Y}_{n-1}  & \widetilde{Z}_{n}\ar[l]^{\rho_{n}} & \widetilde{Y}_{n}\ar[l]^{\theta_{n}}	
	}
	\end{equation*}
Let $T=T_1\times T_2$, and $$\tau_m=T\times T^2\times\ldots\times T^m, \ m\in \N,$$
$$\tau_m'={\rm id}_X\times \tau_{m-1}, \ m\in \N.$$
We want to prove that $(\widetilde{Y}_n,T)$ is an $n$-step topological characteristic factor of $(\widetilde{X}_{n},T)$.
That is, there is some dense $G_\delta$ subset $\Omega$ of $(\widetilde{X}_n, T)$ such that for all ${\bf x}\in \Omega$, $\overline{\O}({\bf x}, \tau_{n})$ is $\pi_{n}^{(n)}$-saturated.
	
To this aim, we prove that for every $1\leq m\leq n$, $(\widetilde{Y}_m,T)$ is an $m$-step topological characteristic factor of $(\widetilde{X}_n,T)$.	We prove by induction on $m$.

\medskip

First we show the case $m=1$, that is, $(\widetilde{Y}_1,T)$ is a $1$-step topological characteristic factor of $(\widetilde{X}_n,T)$. Recall that we have the following diagrams:
\begin{equation*}
	\xymatrix{
		((X_1)_n,T_1)\ar[d]_{\lambda^1_{0}}\ar[dr]^{\lambda^1_{1}}   \\
		((Y_1)_{0},T_1) & ((Y_1)_{1},T_1)\ar[l],
	}\quad \quad
\xymatrix{
		((X_2)_n,T_2)\ar[d]_{\lambda^2_{0}}\ar[dr]^{\lambda^2_{1}}   \\
		((Y_2)_{0},T_2) & ((Y_2)_{1},T_2),\ar[l],
	}
	\end{equation*}	
where $\lambda_0^1,\lambda_1^1,\lambda_0^2,\lambda_1^2$ are RIC.
Also we have the following diagram:
\begin{equation*}
	\xymatrix{
		(\widetilde{X}_n,T)\ar[d]_{\lambda_{0}}\ar[dr]^{\lambda_{1}}   \\
		(\widetilde{Y}_0,T) & (\widetilde{Y}_1,T),\ar[l],
	}
	\end{equation*}	
where $\widetilde{X}_n=(X_1)_n\times (X_2)_n, \widetilde{Y}_0=(Y_1)_0\times (Y_2)_0,\widetilde{Y}_1=(Y_1)_1\times (Y_2)_1$, and $T=T_1\times T_2$, $\lambda_0=\lambda_0^1\times \lambda_0^2$, $\lambda_1=\lambda_1^1\times \lambda_1^2$. To show that $(\widetilde{Y}_1,T)$ is a $1$-step topological characteristic factor of $(\widetilde{X}_n,T)$, means that there exists a dense $G_\d$ set $\Omega_1$ of $\widetilde{X}_n$ such that for each ${\bf x}\in \Omega_1$ the orbit closure $\overline{\O}({\bf x},T)$ is $\lambda_1$-saturated. To this aim, we need to use Lemma \ref{lem-Glasner}.

Note that $N_1(\widetilde{X}_n)=\overline{\O}(\Delta_{1}(\widetilde{X}_n),\tau_{1})
=\overline{\O}(\widetilde{X}_n,T)=\widetilde{X}_n$, and $N_1(\widetilde{Y}_0)=\widetilde{Y}_0$, $N_1(\widetilde{Y}_1)=\widetilde{Y}_1$.
Since proximal extension preserves Ellis groups, the conditions  ``$B_i=H(F_i)A_i$" in Lemma \ref{lem-Glasner} are satisfied.
Since $((X_i)_n,T_i),i=1,2$ are minimal, $N_1(\widetilde{X}_n)=\widetilde{X}_n=(X_1)_n\times (X_2)_n$ has a dense set of $T$-minimal points. Now we take $Q=N_1(\widetilde{X}_n)$ in Lemma \ref{lem-Glasner}, and then conditions (2), (3) in Lemma \ref{lem-Glasner} are trivial. Thus by Lemma \ref{lem-Glasner}, there exists a dense $G_\delta$ subset $\Omega_1$ of $Q=\widetilde{X}_n$ such that ${\bf x}\in \Omega_1$ implies $\overline{\O}({\bf x})$ is $\lambda_1$-saturated. That is, $(\widetilde{Y}_1,T)$ is a $1$-step topological characteristic factor of $(\widetilde{X}_n,T)$.

\medskip

Assume that for $1\le m\le n-1$, $(\widetilde{Y}_{m},T)$ is an $m$-step topological characteristic factor of $(\widetilde{X}_n,T)$. We will show that $(\widetilde{Y}_{m+1},T)$ is an $(m+1)$-step topological characteristic factor of $(\widetilde{X}_n,T)$.

Let $M$ be the lowest common multiple of $1,2,\ldots,m+1$. Assume the decompositions of $X_i,1\leq i\leq 2$ as in Lemma \ref{lem3} respect to $T_i^M$ are:
	\begin{equation*}
	X_i= X^{M,1}_i\cup\ldots\cup X^{M, l_M^i}_i,\ i=1,2,
	\end{equation*}
where $l_M^i|M$, $T_iX_i^{M,j}= X_i^{M,j+1 \ ({\rm mod} \ l_M^i)}$, and the systems $(X_i^{M,j},T_i^M), j=1,\ldots,l_M^i$ are minimal.

By Lemma \ref{thm2.4}, for $i=1,2$ and for each $\nu\le n$, for systems appeared in (\ref{s8}) we also have the decompositions:
\begin{equation*}
	(X_i)_\nu=(X_i)_\nu^{M,1}\cup \ldots \cup (X_i)_\nu^{M,l_M^i},\ 0\le \nu\le n,
\end{equation*}
\begin{equation*}
	(Y_i)_\nu=(Y_i)_\nu^{M,1}\cup \ldots \cup (Y_i)_\nu^{M,l_M^i}, \ 1\le \nu\le n,
\end{equation*}
\begin{equation*}
	(Z_i)_\nu=(Z_i)_\nu^{M,1}\cup \ldots \cup (Z_i)_\nu^{M, l^i_M}, \ 1\le \nu\le n,
\end{equation*}
such that
\begin{enumerate}
  \item  the systems $((H_i)_\nu^{M,j},T_i^M), j=1,\ldots,l_M^i$ are minimal, where $H$ is one of $X,Y,Z$.
  \item $\pi^i_\nu\left((X_i)_\nu^{M,j}\right)=(Y_i)_\nu^{M,j};\sigma^i_{\nu+1}\left((X_i)_\nu^{M,j}\right)
      =(Z_i)_{\nu+1}^{M,j}; \rho^i_{\nu+1}
      \left({(Z_i)_{\nu+1}^{M,j}}\right )=(Y_i)_\nu^{M,j},\ 1\leq j\leq l_M^i$.
  \item the PI-towers of $\pi^i:(X_i^{M,j},T_i^M)\rightarrow(\pi(X_i^{M,j}),T_i^M), 1\leq j\leq l_M^i$ are:
\end{enumerate}
\begin{equation*}\small
	\xymatrix{
	X_i^{M,j}\ar[d]_{\pi^i}&	(X_i)^{M,j}_0\ar[l]_{\tilde\theta^i_0}\ar[d]_{\pi^i_0}\ar[dr]^{\sigma^i_1} & & (X_i)^{M,j}_1\ar[d]_{\pi^i_1}\ar[ll]_{\tilde\theta_1^i}\ar@{}[r]|{\cdots} & (X_i)^{M,j}_{n-1}	\ar[d]_{\pi^i_{n-1}}\ar[dr]^{\sigma^i_{n}} & & (X_i)^{M,j}_{n}\ar[d]_{\pi_{n}^i}\ar[ll]_{\tilde\theta_{n}^i}	\\
	\pi^i(X_i^{M,j}) & \pi^i_0((X_i)_0^{M,j})\ar[l]^{\theta_0^i} & (Z_i)^{M,j}_1\ar[l]^{\rho^i_1} & (Y_i)^{M,j}_1\ar[l]^{\theta^i_1}	\ar@{}[r]|{\cdots} & (Y_i)^{M,j}_{n-1}  & (Z_i)^{M,j}_{n}\ar[l]^{\rho^i_{n}} & (Y_i)^{M,j}_{n}\ar[l]^{\theta^i_{n}}	
	}
	\end{equation*}

Note that $$\widetilde{X}_n=(X_1)_n\times (X_2)_n=\bigcup_{1\le j_1 \le l^1_M \atop{1\le j_2\le l_M^2}}
(X_1)_n^{M,j_1}\times (X_2)^{M,j_2}_n.$$
We will show that for every $1\leq j_1\leq l_M^1,1\leq j_2\leq l_M^2$ there is a dense $G_\delta$ subset $\Omega_{j_1,j_2}$ of $(X_1)^{M,j_1}_n\times (X_2)^{M,j_2}_n$ which is $\lambda_{m+1}$-saturated. Then
	\begin{equation*}
	\Omega_{m+1}=\bigcup_{1\le j_1 \le l^1_M \atop{1\le j_2\le l_M^2}}\Omega_{j_1,j_2}.
	\end{equation*}
is what we are looking for.

\medskip

We show the case $j_1=1,j_2=1$, and the proof for general case is the same. That is, we only show how to find $\Omega_{1,1}$ in $(X_1)_n^{M,1}\times (X_2)^{M,1}_n$. To prove this, we will use Lemma \ref{lem-Glasner}. So it needs to verify conditions in Lemma \ref{lem-Glasner}.

First note that $((X_1)_n^{M,1},T_1^M)$ and $((X_2)_n^{M,1},T_2^M)$ are minimal. By Remark \ref{rem-decomposition}-(3), for every $1\leq j \leq m+1$, as $j | M$, there is some $T_1^j$-minimal component $((X_1)_n^{j, 1},T_1^j)$ of $(X_1)_n$ such that $(X_1)_n^{M,1} \subseteq (X_1)_n^{j,1}$. Similarly there is some $T_2^j$-minimal component $((X_2)_n^{j, 1},T_2^j)$ of $(X_2)_n$ such that $(X_2)_n^{M,1} \subseteq (X_2)_n^{j,1}$.

For $1\leq j\leq m+1,\ i=1,2, $ we have the following communicative diagram:
	\begin{equation*}
	\xymatrix{
		((X_i)_n^{j,1},T_i^j)\ar[d]_{\lambda^i_{m}}\ar[dr]^{\lambda^i_{m+1}}   \\
		((Y_i)_{m}^{j,1},T_i^j) & ((Y_i)_{m+1}^{j,1},T_i^j)\ar[l],
	}
	\end{equation*}	
where $(Y_i)_{m}^{j,1}=\lambda^i_{m}\left((X_i)_n^{j,1}\right)$ and $(Y_i)_{m+1}^{j,1}=\lambda^i_{m+1}\left((X_i)_n^{j,1}\right)$. Since proximal extension preserves Ellis groups, by Lemma \ref{thm2.4}, we conclude that the conditions  ``$B_i=H(F_i)A_i$" in Lemma \ref{lem-Glasner} are satisfied. Let
	\begin{equation*}
	N_{m+1}((X_1)_n^{M,1}\times (X_2)^{M,1}_n)=\overline{\O}(\Delta_{m+1}((X_1)_n^{M,1}\times (X_2)^{M,1}_n),\tau_{m+1}).
	\end{equation*}
	Since $((X_i)_n,T_i),i=1,2$ are minimal, $\widetilde{X}_n=(X_1)_n\times (X_2)_n$ has a dense set of $T$-minimal points. By Lemma \ref{lemma4.7},
	\begin{equation*}
	N_{m+1}(\widetilde{X}_n)=\overline{\O}(\Delta_{m+1}(\widetilde{X}_n),\tau_{m+1})
	\end{equation*}
	has a dense set of $\tau_{m+1}$-minimal points.
	By Lemma \ref{lem6}, $N_{m+1}((X_1)_n^{M,1}\times (X_2)^{M,1}_n)$
	is clopen in
	$N_{m+1}(\widetilde{X}_n)$.
	Hence $N_{m+1}((X_1)_n^{M,1}\times (X_2)^{M,1}_n)$ has a dense set of $\tau_{m+1}$-minimal points.
	
By induction hypothesis,
	let  $\Omega_0$ be the dense $G_\delta$ subset  of $\widetilde{X}_n=(X_1)_n\times (X_2)_n$ such that for any ${\bf x}\in \Omega_0,$
	\begin{equation*}
	\overline{\O}({\bf x}^{(m)},\tau_{m})
	\end{equation*}
	is $\lambda_{m}^{(m)}=(\lambda_{m}^1\times \lambda_{m}^2)^{(m)}$-saturated. Recall that $\widetilde{Y}_{m}=(Y_1)_{m}\times (Y_2)_{m}$.
	
\medskip

\noindent	\textbf{Claim:}\ {\em $
	N_{m+1}(\widetilde{X}_n)=(\lambda_{m}^{(m+1)})^{-1}(N_{m+1}(\widetilde{Y}_{m})).$}

\medskip
	
\begin{proof}[Proof of claim] \renewcommand{\qedsymbol}{}It is clear that
	\begin{equation*}
	N_{m+1}(\widetilde{X}_n)\subseteq (\lambda_{m}^{(m+1)})^{-1}(N_{m+1}(\widetilde{Y}_{m})).
	\end{equation*}
Now we prove the converse.
	Let ${\bf y}\in \widetilde{Y}_{m}$
and ${\bf x}	\in \lambda_{m}^{-1}({\bf y})$. Since $\Omega_0$ is dense, choose $\{{\bf x}_i\}_{i\in\N}\subseteq \Omega_0$ such that ${\bf x}_i\rightarrow {\bf x}, i\to\infty$. Let ${\bf y}_i=\lambda_{m}({\bf x}_i)$, then ${\bf y}_i\rightarrow {\bf y}, i\to\infty.$
	Since $\{{\bf x}_i\}_{i\in\N}\subseteq \Omega_0$, for each $i\in \N$
	\begin{equation*}
\overline{\O}(({\bf x}_i)^{(m)},\tau_{m})=(\lambda_{m}^{(m)})^{-1}\left(\overline{\O}({\bf y_i}^{(m)},\tau_{m})\right).
	\end{equation*}
It follows that
\begin{equation*}
\begin{split}	
& \quad \{{\bf x}_i\}\times(\lambda_{m}^{(m)})^{-1}\left(\overline{\O}({\bf y_i}^{(m)},\tau_{m})\right)\\
	&=\{{\bf x}_i\}\times \overline{\O}({\bf x}_i^{(m)},\tau_{m}) \\
	&=\overline{\O}({\bf x}_i^{(m+1)},{\rm id}_X \times \tau_{m})=\overline{\O}({\bf x}_i^{(m+1)},\tau'_{m+1})\\
	&\subseteq \overline{\O}(\Delta_{m+1}(\widetilde{X}_n), \tau'_{m+1})=\overline{\O}(\Delta_{m+1}(\widetilde{X}_n),\tau_{m+1}).
\end{split}		
\end{equation*}
In particular,
	\begin{equation*}
	\{{\bf x}_i\}\times(\lambda_{m}^{(m)})^{-1}\left({\bf y_i}^{(m)}\right)\subseteq \overline{\O}(\Delta_{m+1}(\widetilde{X}_n),\tau_{m+1}).
	\end{equation*}
Note that $\lambda_{m}^{-1}$ is continuous as $\lambda_{m}$ is open, and we have
	\begin{equation*}
	\begin{split}
	&\quad \{{\bf x}\}\times(\lambda_{m}^{(m)})^{-1}\left({\bf y}^{(m)}\right)=\lim_{i\to\infty}\{{\bf x}_i\}\times(\lambda_{m}^{(m)})^{-1}\left({\bf y_i}^{(m)}\right)  \\
	&\subseteq \overline{\O}(\Delta_{m+1}(\widetilde{X}_n),\tau_{m+1}).
	\end{split}
	\end{equation*}
That is,
	\begin{equation*}
	(\lambda_{m}^{(m+1)})^{-1}\left({\bf y}^{(m+1)}\right)= \lambda^{-1}_{m}({\bf y})\times(\lambda_{m}^{(m)})^{-1}\left({\bf y}^{(m)}\right)  \subseteq \overline{\O}(\Delta_{m+1}(\widetilde{X}_n),\tau_{m+1}).
	\end{equation*}
	Since ${\bf y}\in \widetilde{Y}_{m}$ is arbitrary, we have
	\begin{equation*}
	(\lambda_{m}^{(m+1)})^{-1}(\Delta_{m+1}(\widetilde{Y}_{m}))
	\subseteq \overline{\O}(\Delta_{m+1}(\widetilde{X}_n),\tau_{m+1}).
	\end{equation*}
	By the continuity of $\lambda_{m}^{-1}$, we have
	\begin{equation*}
	\begin{split}
	&\quad (\lambda_{m}^{(m+1)})^{-1}(N_{m+1}(\widetilde{Y}_{m}))= (\lambda_{m}^{(m+1)})^{-1}(\overline{\O}(\Delta_{m+1}(\widetilde{Y}_{m})
	 ,\tau_{m+1})) \\
	&=\overline{\O}((\lambda_{m}^{(m+1)})^{-1}(\Delta_{m+1}(\widetilde{Y}_{m}))
	 ,\tau_{m+1})\\
	&\subseteq \overline{\O}(\Delta_{m+1}(\widetilde{X}_n),\tau_{m+1}).
	\end{split}
	\end{equation*}
The proof of the claim is completed.
\end{proof}
	
	Combining the claim, Lemma \ref{thm2.4} and Lemma \ref{lem6}, we conclude that
$$N_{m+1}((X_1)_n^{M,1}\times (X_2)^{M,1}_n)=(\lambda_{m}^{(m+1)})^{-1}(N_{m+1}((Y_1)_n^{M,1}\times (Y_2)^{M,1}_n).$$
	
Let $U_1\times U_2$ be a non-empty open subset of $(X_1)_n^{M,1}\times (X_2)^{M,1}_n$.
Since $((X_i)_n^{M,1},T_i^M),i=1,2$, are minimal, we can find $N\in\N$ such that
\begin{equation*}
	(X_i)_n^{M,1}=\bigcup_{j=1}^{N} T_i^{Mj} U_i, \quad i=1,2.
\end{equation*}
Hence
	\begin{equation*}
	(X_1)_n^{M,1}\times (X_2)^{M,1}_n=\bigcup_{1\leq t_1,t_2\leq N}(T_1^{Mt_1}\times T_2^{Mt_2})(U_1\times U_2).
	\end{equation*}
So
	\begin{equation*}
	\Delta_{m+1}((X_1)_n^{M,1}\times (X_2)^{M,1}_n)=\bigcup_{1\leq t_1,t_2\leq N}(T_1^{Mt_1}\times T_2^{Mt_2})^{(m+1)}\Delta_{m+1}(U_1\times U_2),
	\end{equation*}
	where we recall that
	\begin{equation*}
	\Delta_{m+1}(U_1\times U_2)=\{{\bf x}^{(m+1)}: {\bf x}\in U_1\times U_2\}.
	\end{equation*}
	
	It follows that
	\begin{equation*}
	\begin{split}
	 & \quad N_{m+1}((X_1)_n^{M,1}\times (X_2)^{M,1}_n)= {\O}(\Delta_{m+1}((X_1)_n^{M,1}\times (X_2)^{M,1}_n),\tau_{m+1})\\
	&=\bigcup_{1\leq t_1,t_2\leq N}(T_1^{Mt_1}\times T_2^{Mt_2})^{(m+1)}{\O}(\Delta_{m+1}(U_1\times U_2),\tau_{m+1}).
	\end{split}
	\end{equation*}
By Baire's theorem, we have that
	\begin{equation*}
	{\rm int}_{N_{m+1}((X_1)_n^{M,1}\times (X_2)^{M,1}_n)}(\overline{\O}(\Delta_{m+1}(U_1\times U_2),\tau_{m+1}))\neq\emptyset.
	\end{equation*}

\medskip
	
Now we have verified all conditions in Lemma \ref{lem-Glasner}, and hence there exists a dense $G_\delta$ subset $\Omega_{1,1}$ of $(X_1)_n^{M,1}\times (X_2)^{M,1}_n$ such that ${\bf x}\in \Omega_{1,1}$ implies $\overline{\O}({\bf x}^{(m+1)},\tau_{m+1})$ is $\lambda_{m+1}^{(m+1)}$-saturated.

Similarly, for every $1\leq j_1\leq l_M^1,1\leq j_2\leq l_M^2$ there is such a dense $G_\delta$ subset $\Omega_{j_1,j_2}$ of $(X_1)^{M,j_1}_n\times (X_2)^{M,j_2}_n$. Let
	\begin{equation*}
	\Omega_{m+1}=\bigcup_{1\le j_1 \le l^1_M,1\le j_2\le l_M^2}\Omega_{j_1,j_2}.
	\end{equation*}
It is easy to see that  $\Omega_{m+1}$ is a dense $G_\delta$ subset of $\widetilde{X}_n=(X_1)_n\times (X_2)_n$ and for each ${\bf x}\in \Omega_{m+1}$ we have that $\overline{\O}({\bf x}^{(m+1)},\tau_{m+1})$ is $\lambda_{m+1}^{(m+1)}$-saturated. Hence
	 $(\widetilde{Y}_{m+1},T)$ is an $(m+1)$-step topological characteristic factor of $(\widetilde{X}_n,T)$. By induction, we have proved that for every $1\leq m\leq n$, $(\widetilde{Y}_m,T)$ is an $m$-step topological characteristic factor of $(\widetilde{X}_n,T)$. In particular, $(\widetilde{Y}_n,T)$ is an $n$-step topological characteristic factor of $(\widetilde{X}_{n},T)$.
The proof of the first part of Theorem \ref{main} is completed.

\medskip

If in addition, $(\widetilde{X}_n, T)$ is transitive. Then we start with $m=0$, and in this case we consider $\lambda_0: \widetilde{X}_n\rightarrow \widetilde{Y}_0$. Since $(\widetilde{X}_n, T)$ is transitive, the set of transitive points $\Omega_0$ is a dense $G_\d$ set of $\widetilde{X}_n$. For all ${\bf x}\in \Omega_0$, $\overline{\O}({\bf x},T)=\widetilde{X}_n$ and it is $\lambda_0$-saturated. That is, $(\widetilde{Y}_0,T)$ is a $1$-step topological characteristic factor of $(\widetilde{X}_n,T)$. The rest of proof is the same with the proof of the first part of the theorem, and the only difference is that now $(\widetilde{Y}_n,T)$ is an $(n+1)$-step topological characteristic factor of $(\widetilde{X}_{n},T)$.
\end{proof}

\section{Independent pair along arithmetic progressions}\label{section-ind}

In this section we will study the independence pair along arithmetic progressions. We show that if a minimal system has no nontrivial independent pair along arithmetic progressions, then it is a PI-system; and if a minimal system has no nontrivial independent pair along arithmetic progressions of order $d$, then up to a canonically defined proximal extension, it is equal to it's canonical PI flow of order $d$. In particular, for a distal minimal system, if it has no nontrivial independent pair along arithmetic progressions of order $d$, then it is equal to its largest distal factor of order $d$.

\subsection{Regionally proximal pair along arithmetic progressions}\
\medskip

The regionally proximal pair of order $d$ along arithmetic progressions
was introduced and studied in \cite{GHSY}.

\begin{de}\label{arithm} Let $(X,T)$ be a t.d.s. and $d\in\N$.
	We say that $(x,y)\in X\times X$ is a {\em regionally proximal pair of order
		$d$ along arithmetic progressions} if for each $\d>0$ there exist
	$x',y'\in X$ and $n\in\Z$ such that $\rho(x, x') < \d,
	\rho(y, y') <\d$ and $$\rho(T^{in}(x'),
	T^{in}(y'))<\d\ \text{for each}\ 1\le i\le d.$$
	
	The set of all such
	pairs is denoted by $\AP^{[d]}(X)$ and is called the {\em regionally
		proximal relation of order $d$ along arithmetic progressions}.
\end{de}

\begin{thm}\cite{GHSY}
	Let $(X,T)$ be a t.d.s. with $\AP^{[d]}(X)=\Delta(X)$ for some integer
	$d\geq 1$, then for each ergodic Borel measure $\mu$, $(X,\mu,T)$ is
	measure theoretical isomorphic to a $d$-step pro-nilsystem.
	
\end{thm}

In \cite{GHSY} it is shown that if $(X,T)$ is a strictly ergodic distal system with the property that the maximal
topological and measurable $d$-step pro-nilsystems are isomorphic,
then $\AP^{[d]}$ is an equivalence relation for each $d\in\N$. We do not have general result on this direction.

\subsection{Independence}\
\medskip


Let us recall the definition of independence set.

\begin{de}
	Let $(X, T)$ be a t.d.s. Given a tuple $\A=(A_1,\ldots,A_k)$ of
	subsets of $X$ we say that a subset $F\subseteq\Z$ is an {\em independence set
	} for $\A$ if for any non-empty finite subset
	$J\subseteq F$ and any $s=(s(j): j\in J ) \in\{1,\ldots,k\}^J$ we have
	$$\bigcap\limits_{j\in J}T^{-j}A_{s(j)}\neq\emptyset \ .$$
	We shall denote the collection of all independence sets for $\A$ by
	$\Ind(A_1,\ldots,A_k)$ or $\Ind\A$.
\end{de}




\begin{de}\label{de}
	Let $(X,T)$ be a t.d.s. and $d\in \N$. A pair $(x_1,x_2)\in X\times X$ is called an \textit{$\Ind_{ap}$-pair of order d} (ap for arithmetic progression)  if for every pair of neighborhoods $U_1,U_2$ of $x_1$ and $x_2$ respectively, there exists some $n\in \Z\setminus\{0\}$ such that $\{n, 2n,\ldots,dn\}$ is an independence set for $(U_1,U_2)$, i.e. for each $t=(t_1,t_2,\ldots, t_d)\in \{1,2\}^d$,
	\begin{equation*}
	T^{-n}U_{t_1}\cap T^{-2n}U_{t_2}\cap\ldots\cap T^{-dn}U_{t_d}\neq\emptyset.
	\end{equation*}
	A pair $(x_1,x_2)\in X\times X$ is called an \textit{$\Ind_{ap}$-pair} if for every $d\in \N$,  $(x_1,x_2)$ is an Ind$_{ap}$-pair of order $d$.
\end{de}
Denote by Ind$_{ap}(X,T)$ (resp. Ind$^{[d]}_{ap}(X,T)$) the set of all Ind$_{ap}$-pairs (resp. Ind$_{ap}$-pairs of order $d$) of $(X,T)$.

\begin{rem}\label{rem-7.5}
\begin{enumerate}
          \item Note that for all $d\in \N$, $\Ind_{ap}(X,T)\subseteq \AP^{[d]}(X)$.
          \item Let $\pi: (X,T)\rightarrow (Y,T)$ be a factor map and $d\in \N$. By definition, it is easy to verify that
              $$\pi\times\pi(\Ind_{ap}(X,T))\subseteq \Ind_{ap}(Y,T),
              \ \pi\times\pi(\Ind_{ap}^{[d]}(X,T))\subseteq \Ind_{ap}^{[d]}(Y,T).$$
        \end{enumerate}
\end{rem}

\subsection{A minimal system without nontrivial Ind$_{ap}$-pairs}\
\medskip

In this subsection we study the dynamical properties of a minimal
system without nontrivial ${\rm Ind}_{ap}$-pair.

\medskip

We first give a criterion to be an Ind$_{ap}$-pair.	Recall that we denote $T\times T^2\times\ldots\times T^d$ by $\tau_d$.

\begin{lem}\label{lemma4.15}
	Let $(X,T)$ be a t.d.s. and $(x_1,x_2)\in X\times X\setminus \Delta(X)$. Then $(x_1,x_2)\in \Ind_{ap}(X,T)$ if and only if for every $d\in\N$, $\{x_1,x_2\}^d\subseteq N_d(X)=\overline{\O}(\Delta_d(X),\tau_d)$.
\end{lem}

\begin{proof}
``$\Rightarrow$'': Let $(x_1,x_2)\in$ Ind$_{ap}(X,T)$, and $d\in \N$.  Let $U_1,U_2$ be neighborhoods of $x_1$ and $x_2$ respectively.  By definition there exists some $n\neq 0$ such that for each $(t_1,\ldots,t_d)\in \{1,2\}^d$,
	\begin{equation*}
	T^{-n}U_{t_1}\cap T^{-2n}U_{t_2}\cap\ldots\cap T^{-dn}U_{t_d}\neq\emptyset.
	\end{equation*}
	Let $(x_{t_1},\ldots,x_{t_d})\in\{x_1,x_2\}^d$, where $(t_1,\ldots,t_d)\in \{1,2\}^d$. By property above, choose
	\begin{equation*}
	x\in T^{-n}U_{t_1}\cap T^{-2n}U_{t_2}\cap\ldots\cap T^{-dn}U_{t_d}.
	\end{equation*}
	Then
	\begin{equation*}
	(\tau_d)^n(x^{(d)})=(\tau_d)^n(x,x,\ldots,x)\in (U_{t_1}\times U_{t_2}\times\ldots\times U_{t_d}).
	\end{equation*}
	Since $(U_{t_1}\times U_{t_2}\times\ldots\times U_{t_d})$ is a neighborhood of $(x_{t_1}, x_{t_2},\ldots,x_{t_d})$ and $U_1, U_2$ are arbitrary, we have that
	\begin{equation*}
	(x_{t_1},\ldots,x_{t_d}) \in\overline{\O}(\Delta_d(X),\tau_d).
	\end{equation*}
	Hence $\{x_1,x_2\}^d\subseteq \overline{\O}(\Delta_d(X),\tau_d).$
	
	``$\Leftarrow$'':Let $d\in \N$ and let $U_1,U_2$ be neighborhoods of $x_1$ and $x_2$ respectively. By assumption, $\{x_1,x_2\}^{d2^d}\subseteq \overline{\O}(\Delta_{d2^d}(X),\tau_{d2^d})$.
	Then
	\begin{equation*}
	\prod\limits_{(t_1,\ldots,t_d)\in \{1,2\}^d}(x_{t_1},\ldots,x_{t_d})\in\{x_1,x_2\}^{d2^d}\subseteq \overline{\O}(\Delta_{d2^d}(X),\tau_{d2^d}).
	\end{equation*}
	Note that
	\begin{equation*}
	\prod\limits_{(t_1,\ldots,t_d)\in \{1,2\}^d}(x_{t_1},\ldots,x_{t_d})\in\prod\limits_{(t_1,\ldots,t_d)\in \{1,2\}^d}(U_{t_1}\times\ldots\times U_{t_d}).
	\end{equation*}
	Hence there exist some $x\in X$ and $n\neq 0$ such that
	\begin{equation*}
	(\tau_{d2^d})^n(x,x,\ldots,x)=(\tau_{d2^d})^n(x^{(d2^d)})\in\prod\limits_{(t_1,\ldots,t_d)\in \{1,2\}^d}(U_{t_1}\times\ldots\times U_{t_d}).
	\end{equation*}
	It follows that for each $(t_1,\ldots,t_d)\in \{1,2\}^d$,
	\begin{equation*}
	T^{-n}U_{t_1}\cap T^{-2n}U_{t_2}\cap\ldots\cap T^{-dn}U_{t_d}\neq\emptyset.
	\end{equation*}
	By definition, $(x_1,x_2)\in$ Ind$_{ap}(X,T)$.
\end{proof}

\begin{cor}\label{lemma4.16}
	Let $\pi:(X,T)\rightarrow (Y,T)$ be an open extension of t.d.s. If $(Y,T)$ is a $d$-step topological characteristic factor of $(X,T)$ for all $d\in \N$, then $R_\pi\setminus\Delta(X)\subseteq \Ind_{ap}(X,T)$.
\end{cor}
\begin{proof}
	Let $(x_1,x_2)\in R_\pi\setminus \Delta(X)$. It is easy to see that for all $d\in \N$
	\begin{equation*}
	\{x_1,x_2\}^d\subseteq R^d_\pi=(\pi^{(d)})^{-1}(\Delta_d(Y))\subseteq \overline{\O}(\Delta_d(X),\tau_d).
	\end{equation*}
	By Lemma \ref{lemma4.15}, $(x_1,x_2)\in$ Ind$_{ap}(X,T)$. That is, $R_\pi\setminus \Delta(X)\subseteq \Ind_{ap}(X,T)$.
\end{proof}

\medskip

Let $\pi:(X,T)\rightarrow(Y,T)$ be a factor map of minimal t.d.s. Let
\begin{equation*}
\Ind_{ap}(\pi)=\Ind_{ap}(X,T)\cap R_{\pi}, \quad \Ind_{ap}^{[d]}(\pi)=\Ind_{ap}^{[d]}(X,T)\cap R_{\pi}.
\end{equation*}
The following theorem shows that if $\Ind_{ap}(\pi)$ is trivial, then $\pi$ is a PI extension:
\begin{thm}\label{thm4.4}
Let $\pi:(X,T)\rightarrow(Y,T)$ be a factor map of minimal t.d.s.
If $$\Ind_{ap}(\pi)=\Delta(X),$$	
then $\pi$ is PI.	

In particular, a minimal t.d.s. without nontrivial ${\rm Ind}_{ap}$-pair is a PI-system.
\end{thm}
\begin{proof}
Let the PI-tower of $\pi:(X,T)\rightarrow(Y,T)$ be as follows:
	\begin{equation*}
	\xymatrix{		X\ar[d]_{\pi}&X_0\ar[l]_{\tilde\theta_0}\ar[d]_{\pi_0}\ar[dr]^{\sigma_1} & & X_1\ar[d]_{\pi_1}\ar[ll]_{\tilde\theta_1}\ar@{}[r]|{\cdots} &X_v	\ar[d]_{\pi_v}\ar[dr]^{\sigma_{v+1}} & & X_{v+1}\ar[d]_{\pi_{v+1}}\ar[ll]_{\tilde\theta_{v+1}}	\ar@{}[r]|{\cdots} &X_\eta=X_\infty\ar[d]_{\pi_\infty} \\
		Y&Y_0\ar[l]^{\theta_0}  & Z_1\ar[l]^{\rho_1} & Y_1\ar[l]^{\theta_1}	\ar@{}[r]|{\cdots} &Y_{v}  & Z_{v+1}\ar[l]^{\rho_{v+1}} & Y_{v+1}\ar[l]^{\theta_{v+1}}	\ar@{}[r]|{\cdots} &Y_\eta=Y_\infty.
	}
	\end{equation*}

We have the following claim:

\medskip

\noindent
\textbf{Claim}: {\em  For an ordinal $v_0$ and $(x_{v_0}^1,x_{v_0}^2)\in \Ind_{ap}(\pi_{v_0})\setminus\Delta(X_{v_0})$,
	\begin{enumerate}
		\item 	if $v_0$ is a limit ordinal, then we can find $v<v_0$ such that
		\begin{equation*}
		(x_v^1,x_v^2)\in\Ind_{ap}(\pi_v)\setminus\Delta(X_v),
		\end{equation*}
		where $(x_v^1,x_v^2)$ is the image of $(x_{v_0}^1,x_{v_0}^2)$;
		\item if $v_0$ is a successor ordinal, then
		\begin{equation*}
		(x_{v_0-1}^1,x_{v_0-1}^2)\in \Ind_{ap}(\pi_{v_0-1})\setminus\Delta(X_{v_0-1}),
		\end{equation*}
		where $(x_{v_0-1}^1,x_{v_0-1}^2)$ is the image of $(x_{v_0}^1,x_{v_0}^2)$.
	\end{enumerate}
}
\begin{proof}[Proof of claim] \renewcommand{\qedsymbol}{}
	First we consider the case that $v_0$ is a limit ordinal. Then we can assume $ (x_{v_0}^1,x_{v_0}^2)=(\{x^1_v\}_{v<v_0},\{x^2_v\}_{v<v_0})$. Since  $x^1_{v_0}\neq x^2_{v_0}$, we can find $v<v_0$ such that $x^1_{v}\neq x^2_{v}$. Since $ (x_{v_0}^1,x_{v_0}^2)\in \Ind_{ap}(\pi_{v_0})$,
by Remark \ref{rem-7.5}-(2)
	\begin{equation*}
	(x_v^1,x_v^2)\in\Ind_{ap}(\pi_v)\setminus\Delta(X_v).
	\end{equation*}
	
Now we consider the case that $v_0$ is a successor ordinal, and consider the following diagram:
	\begin{equation*}
	\xymatrix{
		X_{v_0-1}\ar[d]_{\pi_{v_0-1}}\ar[dr]^{\sigma_{v_0}} & & X_{v_0}\ar[d]_{\pi_{v_0}}\ar[ll]_{\tilde\theta_{v_0}}\\
		Y_{v_0-1}  & Z_{v_0}\ar[l]^{\rho_{v_0}} & Y_{v_0}\ar[l]^{\theta_{v_0}}	
		.}
	\end{equation*}
Note that  $X_{v_0}=X_{v_0-1}\vee Y_{v_0}$ in the structure of the PI-tower (see Theorem \ref{RIC}), and we can assume
	\begin{equation*}
	x_{v_0}^1=(x^1_{v_0-1},y^1_{v_0}),x_{v_0}^2=(x^2_{v_0-1},y^2_{v_0}).
	\end{equation*}
Since $\pi_{v_0}(x^1_{v_0})=\pi_{v_0}(x^2_{v_0})$, we have  $y^1_{v_0}=y^2_{v_0}$. Note that $x^1_{v_0}\neq x^2_{v_0}$, and we have $x^1_{v_0-1}\neq x^2_{v_0-1}$. Since the diagram is communicative, we have
	\begin{equation*}	\pi_{v_0-1}(x^1_{v_0-1},x^2_{v_0-1})=\rho_{v_0}\circ\theta_{v_0}(y^1_{v_0},y^2_{v_0}),
	\end{equation*}
and $(x^1_{v_0-1},x^2_{v_0-1})\in R_{\pi_{v_0-1}}$. Since
$(x_{v_0}^1,x_{v_0}^2)\in \Ind_{ap}(X_{v_0},T)$, by Remark \ref{rem-7.5}-(2) we have $(x_{v_0-1}^1,x_{v_0-1}^2)\in \Ind_{ap}(X_{v_0-1},T)$, and
	\begin{equation*}
	(x_{v_0-1}^1,x_{v_0-1}^2)\in  \Ind_{ap}(\pi_{v_0-1})\setminus\Delta(X_{v_0-1}).
	\end{equation*}
The proof of claim is completed.
\end{proof}

If $\pi_\infty$ is not an isomorphism, then we can choose $(x_\infty^1,x_\infty^2)\in R_{\pi_\infty}\setminus\Delta(X_\infty).$ By Corollary \ref{lemma4.16}, $R_{\pi_\infty}\setminus\Delta(X_\infty)\subseteq \Ind_{ap}(X_{\infty},T)$, and hence	
	\begin{equation*}
	(x_\infty^1,x_\infty^2)\in \Ind_{ap}(\pi_\infty)\setminus\Delta(X_\infty).
	\end{equation*}	
Using Claim inductively, we have some  $(x^1,x^2)\in \Ind_{ap}(\pi)\setminus\Delta(X)$, which contradicts with $\Ind_{ap}(\pi)=\Delta(X)$.
\end{proof}

By Remark \ref{rem-7.5}-(1), we have the following corollary.
\begin{cor}
Let $(X,T)$ be a t.d.s.and $d\in \N$. If $\AP^{[d]}(X)=\Delta(X)$, then $(X,T)$ is PI.
\end{cor}

\subsection{A minimal system without nontrivial Ind$_{ap}$-pair of order $d$}\
\medskip

In this subsection we consider a minimal
system without nontrivial ${\rm Ind}_{ap}$-pair of order $d$.
We give an analogous result to Theorem \ref{thm4.4}.
We show that if a minimal
system $(X,T)$ has no nontrivial ${\rm Ind}_{ap}$-pairs of order $d$, then in the PI-tower of $(X,T)$, $\pi_{d}$ is an isomorphism. We state the result in its relative version.

\begin{thm}
	Let $\pi:(X,T)\rightarrow(Y,T)$ be a factor map of minimal  t.d.s. If
	$$\Ind^{[d]}_{ap}(\pi)=\Delta(X),$$
	then in the PI tower of $\pi$, the order is less than $d$.
\end{thm}

\begin{proof}
Assume the complete canonical PI-tower of order $d$ of $\pi$ is
	\begin{equation*}
	\xymatrix{
		X\ar[d]_{\pi}&	{X}_0\ar[l]_{\tilde\theta_0}\ar[d]_{\pi_0}\ar[dr]^{\sigma_1} & & {X}_1\ar[d]_{\pi_1}\ar[ll]_{\tilde\theta_1}\ar@{}[r]|{\cdots} &{X}_{d-1}	\ar[d]_{\pi_{d-1}}\ar[dr]^{\sigma_{d}} & & {X}_{d}\ar[d]_{\pi_{d}}\ar[ll]_{\tilde\theta_{d}}	\\
		Y &{Y}_0\ar[l]^{\theta_0} & {Z}_1\ar[l]^{\rho_1} &{Y}_1\ar[l]^{\theta_1}	\ar@{}[r]|{\cdots} &{Y}_{d-1}  & {Z}_{d}\ar[l]^{\rho_{d}} & {Y}_{d}\ar[l]^{\theta_{d}}.	
	}
	\end{equation*}
Then the maps $\lambda_j:X_{d}\rightarrow Y_j\ (0\leq j\leq d)$  are RIC.
We need to show that $\pi_d$ is an isomorphism.

	By Theorem \ref{main}, $(Y_{d}^{2^d},T^{(2^d)})$ is a $d$-step topological characteristic factor of $(X_{d}^{2^d}, T^{(2^d)})$. 	Hence
	\begin{equation*}	\overline{\O}(\Delta_{d}(X_{d}^{2^d}),\tau_d)=(\pi_{d}^{(d2^d)})^{-1}(\overline{\O}(\Delta_{d}(Y_{d}^{2^d}),\tau_d)).
	\end{equation*}
	Note that here $\tau_d = T^{(2^d)}\times (T^2)^{(2^d)}\times \ldots \times (T^d)^{(2^d)}$, i.e.
	\begin{equation*}
	\tau_d=\underbrace{T\times T\times\ldots\times T}_{2^d\ {\rm times}}\times\underbrace{T^2\times T^2\times\ldots\times T^2}_{2^d\ {\rm times}}\times\ldots\times \underbrace{T^d\times T^d\times\ldots\times T^d}_{2^d\ {\rm times}}.
	\end{equation*}
	If $R_{\pi_{d}}\neq \Delta(X_{d})$, we can find $(z_1,z_2)\in R_{\pi_{d}}$ with $z_1\neq z_2$. Let $\pi_{d}(z_1)=\pi_{d}(z_2)=y$.
	Choose neighborhoods $U_1,U_2$ of $z_1$ and $z_2$ respectively. Enumerate $\{1,2\}^d$ as:
	\begin{equation*}
	\{1,2\}^d=\{(t^i_1,t^i_2,\ldots,t^i_d):1\leq i\leq 2^d\}.
	\end{equation*}
	Note that
	\begin{equation*}
	(\prod_{i=1}^{2^d}U_{t_1^i}\times\prod_{i=1}^{2^d}U_{t_2^i}\times \ldots\times\prod_{i=1}^{2^d}U_{t_d^i})\bigcap (\pi_{d}^{-1}(y))^{d2^d} \neq\emptyset.
	\end{equation*}
	Since $y^{(d2^d)}\in \Delta_{d}(Y_{d}^{2^d})$, we have
	\begin{equation*}
	(\pi_{d}^{-1}(y))^{d2^d}\subseteq (\pi_{d}^{(d2^d)})^{-1}(\Delta_{d}(Y_{d}^{2^d})),
	\end{equation*}
	and it follows that
	\begin{equation*}	(\prod_{i=1}^{2^d}U_{t_1^i}\times\prod_{i=1}^{2^d}U_{t_2^i}\times \ldots\times\prod_{i=1}^{2^d}U_{t_d^i})\bigcap \overline{\O}(\Delta_{d}(X_{d}^{2^d}),\tau_d)\neq\emptyset.
	\end{equation*}
	Hence there exist $n\in\Z\setminus \{0\}$ and ${\bf x}=(x_1,x_2,\ldots,x_{2^d})\in X_{d}^{2^d}$ such that
	\begin{equation*}
	\begin{split}
	(\tau_d)^n({\bf x}^{(d)})
	\in\prod_{i=1}^{2^d}U_{t_1^i}\times\prod_{i=1}^{2^d}U_{t_2^i}\times \ldots\times\prod_{i=1}^{2^d}U_{t_d^i}.
	\end{split}
	\end{equation*}
	For every $(t_1,t_2,\ldots,t_d)\in\{1,2\}^d$, there is some $1\leq i\leq 2^d$ such that $(t_1,t_2,\ldots,t_d)=(t^i_1,t^i_2,\ldots,t^i_d)$.
	Then
	\begin{equation*}
	(T\times T^2\times\ldots\times T^d)^n(x_i,x_i,\ldots,x_i)\in U_{t_1^i}\times U_{t_2^i}\times\ldots\times U_{t_d^i}=U_{t_1}\times U_{t_2}\times\ldots\times U_{t_d}.
	\end{equation*}
	Hence
	\begin{equation*}
	T^{-n}U_{t_1}\cap T^{-2n}U_{t_2}\cap\ldots\cap T^{-dn}U_{t_d}\neq\emptyset.
	\end{equation*}
By definition, 	$(z_1,z_2)\in \Ind^{[d]}_{ap}(X_{d},T)$. Thus $\Ind_{ap}^{[d]}(\pi_d)\setminus \Delta(X_d)\neq \emptyset$.

Consider the following diagram:
	\begin{equation*}
	\xymatrix{
		X_{d-1}\ar[d]_{\pi_{d-1}}\ar[dr]^{\sigma_{d}} & & X_{d}\ar[d]_{\pi_{d}}\ar[ll]_{\tilde\theta_{d}}\\
		Y_{d-1}  & Z_{d}\ar[l]^{\rho_{d}} & Y_{d}\ar[l]^{\theta_{d}}.	
	}
	\end{equation*}
	Since $X_{d}=X_{d-1}\vee Y_{d}$ (this condition still holds after adjusting $\lambda_j$ to be RIC \cite{G94}), repeating the same discussion as in Theorem \ref{thm4.4}, we have that $\Ind_{ap}^{[d]}(\pi_{d-1})\setminus \Delta(X_{d-1})\neq \emptyset$.

Inductively, we have that $\Ind_{ap}^{[d]}(\pi)\setminus \Delta(X)\neq \emptyset$. A contradiction! Thus $\pi_{d}$ is an isomorphism. The proof is completed.
\end{proof}

\begin{cor}
Let $(X,T)$ be a minimal t.d.s. with $\AP^{[d]}(X)=\Delta(X)$ for some 	$d\in \N$, then it has a PI tower of order $d$.
\end{cor}

\begin{cor}\label{cor-7.11}
Let $(X,T)$ be a distal minimal t.d.s. If $\Ind^{[d]}_{ap}(X,T)=\Delta(X),$	then the order of $(X,T)$ is less than $d$.
\end{cor}

We conjecture that Corollary \ref{cor-7.11} can be improved as follows:
\begin{conj}
	Let $(X,T)$ be a distal minimal t.d.s. If
	$\Ind^{[d]}_{ap}(X,T)=\Delta(X),$
	then $(X,T)$ is a $d$-step pro-nilsystem.
\end{conj}
	
\section{$\Delta$-transitivity}

In this section we study $\Delta$-transitivity and the local version of $\Delta$-transitivity. We show that the condition $\Ind_{ap}(X,T)$=$X\times X$ ($\Ind_{ap}^{[d]}(X,T)$=$X\times X$) is equivalent to $\Delta$-transitivity ($\Delta$-transitivity of order $d$).  We also study the relative case. The property of $\Delta$-transitivity has been systematically studied, we refer readers to \cite{HLYZ,KO,Mo} for details.

\subsection{$\Delta$-transitivity }\
\medskip

Let $(X,T)$ be a t.d.s. and $d\in \N$. Let $U_0,U_1,\ldots,U_d$ be subsets of $X$. Define
\begin{equation*}
N(U_0,U_1,\ldots,U_d)=\{n\in\N:\bigcap_{i=0}^dT^{-in}U_i\neq\emptyset \}.
\end{equation*}

\begin{de}
Let $(X,T)$ be a t.d.s. and $d\in \N$.
$(X,T)$ is said to be {\em $\Delta$-transitive of order $d$} if for any non-empty open subsets $U_0,U_1,\ldots,U_d$ of $X$, $N(U_0, U_1,\ldots,U_d)\neq\emptyset$.
A closed subset $A$ of $X$ is a {\em $\Delta$-transitive set of order d} if for any non-empty open subsets $U_0,U_1,\ldots,U_d$ of $X$ intersecting $A$, $N(U_0\cap A, U_1,\ldots,U_d)\neq\emptyset$.

A t.d.s. $(X,T)$ is said to be {\em$\Delta$-transitive} if it is $\Delta$-transitive of order $d$ for every $d\in \N$. A closed subset $A$ of $X$ is said to be a {\em $\Delta$-transitive set} if $A$ is a $\Delta$-transitive set of order $d$ for all $d\in \N$.
\end{de}

Denote by $\Delta$-$Trans(X,T)$ (resp. $\Delta_d$-$Trans(X,T)$) the set of all  $\Delta$-transitive sets (resp. $\Delta$-transitive sets of order $d$) of $(X,T)$.

\begin{rem}
	By definition we can see that $\{x\}$ is a $\Delta$-transitive set of order $d$  if and only if  $x$ is a multiple recurrent point of order $d$, i.e. there exists a sequence $\{n_1<n_2<\ldots\}$ of $\N$ such that
	\begin{equation*}
	\lim_{k\rightarrow\infty}T^{in_k}x=x,\ i=1,2,\ldots,d.
	\end{equation*}
\end{rem}

\medskip

To prove the next result, we need a characterization of weakly mixing systems. Recall that $(X,T)$ is weakly mixing, if for any non-empty open sets $U_1,U_2,V_1,V_2$, there is some $n\in \Z$ such that $U_1\cap T^{-n}V_1\neq\emptyset$ and $U_2\cap T^{-n}V_2\neq\emptyset$. In \cite{Pet1}, it was proved that $(X,T)$ is weakly mixing if and only if for any non-empty open sets $U,V$, there is some $n\in \Z$ such that $U\cap T^{-n}U \neq\emptyset$ and $U\cap T^{-n}V\neq\emptyset$.

Now we give the conditions that $\Ind_{ap}(X,T)=X\times X$ and $\Ind^{[d]}_{ap}(X,T)=X\times X$:
\begin{thm}\label{thm8.1}
Let $(X,T)$ be a t.d.s. and $d\geq 2$. Then the following conditions are equivalent:
\begin{enumerate}
		\item $\Ind^{[d+1]}_{ap}(X,T)=X\times X$.
		\item $(X,T)$ is $\Delta$-transitive  of order $d$.
		\item The trivial system is a $d$-step topological characteristic factor of $(X,T)$.
		\item $\Delta_{d}$-$Trans(X,T)$ is a dense $G_\delta$ set in $2^X$.
\end{enumerate}
\end{thm}

\begin{proof}
(1)$\Rightarrow$(2): We prove by induction on $d$. When $d=2$, we first show that $(X,T)$ is weakly mixing. Let $U_0,U_1$ be non-empty open sets in $X$. Since $\text{Ind}^{[3]}_{ap}(X,T)=X\times X$,  there exists $n\in\Z$ such that $$U_0\cap T^{-n}U_0\cap T^{-2n}U_1\neq\emptyset.$$ Hence $(X,T)$ is weakly mixing. Now let $U_0,U_1, U_2$ be non-empty open sets in $X$. Since $(X,T)$ is weakly mixing, $(X,T)$ has no isolated point and there exists $n\in\N$ such that $U_0\cap T^{-n}U_1\neq\emptyset.$ Let $V_1=U_0\cap T^{-n}U_1, V_2=T^{-2n}U_2$. Since $\text{Ind}^{[3]}_{ap}(X,T)=X\times X$, there exists $m\in\N$ such that $$V_1\cap T^{-m}V_1\cap T^{-2m}V_2\neq\emptyset.$$  Hence $m+n\in N(U_0,U_1,U_2)$. Thus $(X,T)$ is $\Delta$-transitive  of order $2$.

Now assume the statement holds for $d$, i.e. Ind$^{[d]}_{ap}(X,T)=X\times X$ implies that $(X,T)$ is $\Delta$-transitive of order $d-1$. We show that Ind$^{[d+1]}_{ap}(X,T)=X\times X$ implies that $(X,T)$ is $\Delta$-transitive of order $d$.

Let $U_0,U_1,\ldots,U_d$ be non-empty open sets in $X$. Since
$\text{Ind}^{[d+1]}_{ap}(X,T)\subseteq \text{Ind}^{[d]}_{ap}(X,T)$,
we have Ind$^{[d]}_{ap}(X,T)=X\times X$. By the induction hypothesis, $(X,T)$ is $\Delta$-transitive of order $d-1$. Let $n\in N(U_0,U_1,\ldots,U_{d-1})$, and	let
	\begin{equation*}
	V_1=U_0\cap T^{-n}U_1\cap\ldots\cap T^{-(d-1)n}U_{d-1},\quad V_2=T^{-dn}U_d.
	\end{equation*}
Since $\text{Ind}^{[d+1]}_{ap}(X,T)=X\times X$, there exists $m\in \N$ such that
	\begin{equation*}
	V_1\cap T^{-m}V_1\cap\ldots\cap T^{-(d-1)m}V_1\cap T^{-dm}V_2\neq\emptyset.
	\end{equation*}
It implies that $m+n\in N(U_0,U_1,\ldots,U_d)$. Thus $(X,T)$ is $\Delta$-transitive  of order $d$.

\medskip
	
(2)$\Rightarrow$(1): We first show that $(X,T)$ is weakly mixing. Let $U_0,U_1$ be non-empty open sets in $X$. Then there exists $n\in\Z$ such that $$U_0\cap T^{-n}U_0\cap T^{-2n}U_1\neq\emptyset.$$ Hence $(X,T)$ is weakly mixing.
	
Let $U_0,U_1,\ldots,U_d,V_0,V_1,\ldots,V_d$ be non-empty open sets in $X$. Since $(X,T)$ is weakly mixing, there exists $n\in\Z$ such that
	\begin{equation*}
	(U_0\times U_1\times\ldots\times U_d)\cap (T\times T\times\ldots\times T)^{-n}(V_0\times V_1\times\ldots\times V_d)\neq\emptyset.
	\end{equation*}
	It is easy to check that
	\begin{equation*}
	N(U_0\cap T^{-n}V_0,U_1\cap T^{-n}V_1,\ldots,U_d\cap T^{-n}V_d)\subseteq N(U_0,U_1,\ldots,U_d)\cap N( V_0,V_1,\ldots,V_d).
	\end{equation*}
	Hence the family
	\begin{equation*}
	\{N(U_0,U_1,\ldots,U_d):U_0,U_1,\ldots,U_d \text{\ be non-empty open sets in } X\}
	\end{equation*}
is a filter base.  Hence for any $(x_1,x_2)\in X\times X$ and any open neighborhoods $U_1\times U_2$ of $(x_1,x_2)$,
	\begin{equation*}	\bigcap_{\sigma\in\{1,2\}^{\{0,1,\ldots,d\}}}N(U_{\sigma(0)},U_{\sigma(1)},\ldots,U_{\sigma(d)})\neq\emptyset.
	\end{equation*}
It follows that $(x_1,x_2)\in \text{Ind}^{[d+1]}_{ap}(X,T)$, i.e. $\Ind_{ap}^{[d+1]}(X,T)=X\times X$.
	
\medskip

	 (2)$\Leftrightarrow$(3): See \cite[Proposition 3.1]{HLYZ}.
\medskip
	
	(2)$\Rightarrow$(4): We first show that $\Delta_{d}$-$Trans(X,T)$ is a  $G_\delta$ set in $2^X$. We follow the methods in \cite{HLYZ}. Let
	\begin{equation*}
	\begin{split}
	A_m(X,T)=\{&E\in2^X: \text{there exist}\ k\in\N\ \text{and open  sets}\  W_1,\ldots,W_k\ \text{with} \ {\it diam}\ ( W_i)<\frac{1}{m}\ \\
	& \text{such that}\ E\subseteq \bigcup_{i=1}^k W_i\ \text{ and for any}\ \sigma\in \{1,2,\ldots,k\}^{\{0,1,\ldots,d\}},\\
	&N(E\cap W_{\sigma(0)},W_{\sigma(1)},\ldots,W_{\sigma(d)})\neq\emptyset\ \}
	\end{split}
	\end{equation*}
	It is easy to check that $A_m(X,T)$ is open in $2^X$ and
	\begin{equation*}
	\Delta_{d}\text{-}Trans(X,T)=\bigcap_{m=1}^\infty A_m(X,T).
	\end{equation*}
	Hence $\Delta_{d}$-$Trans(X,T)$ is a  $G_\delta$ set in $2^X$.
	
For any open set $<U_1,U_2,\ldots,U_m>$ in $2^X$,  since $(X,T)$ is $\Delta$-transitive  of order $d$,
	 $$X_0=\bigcap_{V_1,\ldots,V_d\in\B}\bigcup_{n\in\N}(T^{-n}V_1\cap\ldots\cap T^{-dn}V_d)$$ is a dense $G_\delta$ set in $X$, where $\B$ is a countable topological basis of $X$.
	
	 Choose $x_i\in U_i\cap X_0$.
	Hence $$\{x_1,x_2\ldots,x_m\}\in \langle U_1,U_2,\ldots,U_m\rangle,$$ and it is easy to check that $\{x_1,x_2\ldots,x_m\}$ is a $\Delta$-transitive  set of order $d$.
	
\medskip

	(4)$\Rightarrow$(2): Let $U_0,U_1,\ldots,U_d$ be non-empty open sets in $X$. Note that $X\in\bigcap\limits_{i=0}^d \langle U_i,X\rangle,$  and hence $\bigcap\limits_{i=0}^d \langle U_i,X\rangle$ is a non-empty open set in $2^X$.
	Let $A\in\bigcap\limits_{i=0}^d\langle U_i,X\rangle$ be a $\Delta$-transitive  set of order $d$, then $A\cap U_i\neq\emptyset, i=0,1,\ldots,d.$ It follows that $N(A\cap U_0,U_1,\ldots,U_d)\neq\emptyset$. In particular, $N( U_0,U_1,\ldots,U_d)\neq\emptyset.$
\end{proof}

Similarly we have:

\begin{thm}\label{thm8.2}
Let $(X,T)$ be a t.d.s. Then the following conditions are equivalent:
	\begin{enumerate}
		\item $\Ind _{ap}(X,T)=X\times X$.
		\item $(X,T)$ is $\Delta$-transitive.
		\item The trivial system is a $d$-step topological characteristic factor of $(X,T)$ for all $d\ge 2$.
		\item $\Delta$-$Trans(X,T)$ is a dense $G_\delta$ set in $2^X$.	
	\end{enumerate}
\end{thm}

\begin{rem}
If $(X,T)$ is minimal we may say more. In fact in \cite{GHSY}, it is shown that if $(X,T)$ is a minimal t.d.s., then the following
statements are equivalent:
\begin{enumerate}
\item $\Ind_{ap}(X)=X\times X$.

\item $(X,T)$ is weakly mixing.

\item $\AP^{[d]}(X)=X\times X$ for some $d\ge 2$.
\end{enumerate}
\end{rem}

We next give an example, which is adapted from the one in \cite[Proposition 3]{Mo}:
\begin{exam}
There exists a t.d.s. $(X,T)$ with $\Ind^{[d]}_{ap}(X,T)=X\times X$ and $\Ind^{[d+1]}_{ap}(X,T)\subseteq \Delta(X)$.

\medskip
Let $(\{0,1\}^\Z,T)$ be the full shift. Let $X\subseteq\{0,1\}^\Z$ be the collection of all $x\in\{0,1\}^\Z$ satisfying the following  condition:

\medskip
\noindent {\em if $v_1,v_2\ldots,v_{d}$ are  words (may be empty) over $\{0,1\}$ with $1v_11v_21\ldots1v_{d}1$ appearing in $x$, then there must be  some $v_i,v_j$ such that $v_i,v_j$ have different lengths.}
\medskip

\noindent It is easy to see that $X$ is non-empty, closed and $T$-invariant. That is, $(X,T)$ is a subsystem of $(\{0,1\}^\Z,T)$. We will show that $\Ind^{[d]}_{ap}(X,T)=X\times X$ and $\Ind^{[d+1]}_{ap}(X,T)\subseteq \Delta(X)$.

\medskip

We first show that $(X,T)$ is $\Delta$-transitive of order $d-1$. For any non-empty open sets $U_0,U_1,\ldots,U_{d-1}$ in $X$, assume that there exist $m\in\N$ and $x_i\in X,i=0,1,\ldots,d-1$ such that $B(x_i,\frac{1}{m})\subseteq U_i, i=0,1,\ldots,d-1$. Choose $n>m$ and let words
\begin{equation*}
u_i=x_i(-m)\ldots x_i(-1)x_i(0)x_i(1)\ldots x_i(m), \quad i=0,1,\ldots,d-1.
\end{equation*}
Let $$x=0^\infty u_00^n u_10^nu_2\cdots 0^nu_{d-1}0^\infty, x(0)=x_0(0).$$  It is easy to see that $x\in X$. Furthermore
\begin{equation*}
x\in U_0\cap T^{-(2m+1+n)}U_1\cap T^{-2(2m+1+n)}U_2\cap\ldots\cap T^{-(d-1)(2m+1+n)}U_{d-1}.
\end{equation*}
Hence $(X,T)$ is $\Delta$-transitive  of order $d-1$, i.e. $\Ind^{[d]}_{ap}(X,T)=X\times X$.

\medskip

Now we show that $\Ind^{[d+1]}_{ap}(X,T)\subseteq \Delta(X)$. Let $x_0,x_1\in X$ such that $x_0\neq x_1$. Since $\Ind^{[d+1]}_{ap}(X,T)$ is $T\times T$-invariant we can assume $x_0(0)=0,x_1(0)=1$.
Let $$U_0=\{x\in X: x(0)=0 \}, \quad U_1=\{x\in X: x(0)=1 \}.$$ If $(x_0,x_1)\in\Ind^{[d+1]}_{ap}(X,T)$, then there exists $n\in \N$ such that
\begin{equation*}
U_0\cap T^{-n}U_0\cap\ldots\cap T^{-dn}U_0\neq\emptyset,
\end{equation*}
\begin{equation*}
U_0\cap T^{-n}U_1\cap\ldots\cap T^{-dn}U_1\neq\emptyset.
\end{equation*}
Let
\begin{equation*}
x\in U_0\cap T^{-n}U_0\cap\ldots\cap T^{-dn}U_0.
\end{equation*}
Then $x(0)=x(n)=\ldots=x(dn)=1$. Hence $1v_11v_21\ldots1v_{d}1$ appears in $x$, where the length of $v_i,\ i=0,1,\ldots,d$, is $n-1$. It contradicts with that $x\in X$. Hence $(x_0,x_1)\notin\Ind^{[d+1]}_{ap}(X,T)$. So we have that $\Ind^{[d+1]}_{ap}(X,T)\subseteq \Delta(X)$.
\end{exam}

\medskip


\subsection{$\Delta$-transitive sets}\
\medskip

Now we show if $(X,T)$ has a nontrivial $d$-step topological characteristic factor, then there exist lots of $\Delta$-transitive sets of order $d$. To show this we need the following  well known result.

\begin{thm}\cite[Proposition 3.1]{V70}\label{lem9}
Let $\pi:(X,T)\rightarrow (Y,T)$ be a factor map of minimal t.d.s.  Let $R$ be a dense $G_\delta$ subset of $X$, then
\begin{equation*}
	Y_0=\{y\in Y: \pi^{-1}(y)\cap R \ is\ dense \ G_\delta \ in\ \pi^{-1}(y) \}
\end{equation*}
is a dense $G_\delta$ subset of $Y$.
\end{thm}

\begin{thm}\label{thm7.2}
Let $\pi:(X,T)\rightarrow (Y,T)$ be a nontrivial factor map of minimal t.d.s. and $d\ge 2$. If $(Y,T)$ is a $d$-step topological characteristic factor of $(X,T)$, then there exists a dense $G_\delta$ subset $\Omega$ of $X$ such that for any $x\in\Omega, \pi^{-1}(\pi(x))$ is a $\Delta$-transitive set of order $d$ in $(X,T)$.
\end{thm}

\begin{proof}
By the definition of topological characteristic factor,
there exists  a  dense $G_\delta$ subset $\Omega_1$ of $X$ such that for any $x\in\Omega_1$, $\overline{\O}(x^{(d)},\tau_d)$ is $\pi^{(d)}$-saturated.
Since $(X,T)$ is minimal, the set of multiple recurrent points is residual, i.e.
$$\Omega_2=\{x\in X: \tau_d^{n_i}(x^{(d)})\to x^{(d)}, i\to\infty, \ \text{for some sequence}\ \{n_i\}_{i\in \N}\}$$
is a dense $G_\d$ set of $X$ \cite{Leibman94}.
Take $R=\Omega_1\cap\Omega_2$ in Theorem \ref{lem9}, and there is a dense $G_\delta$ subset $Y_0$ of $Y$:
\begin{equation*}
Y_0=\{y\in Y:\pi^{-1}(y)\cap R \ is\ dense \ G_\delta \ in\ \pi^{-1}(y) \}.
\end{equation*}
Let $\Omega=R\cap\pi^{-1}(Y_0)$. Then $\Omega$ is a  dense $G_\delta$ subset of $X$. We will show that for any $x\in\Omega, \pi^{-1}(\pi(x))$ is a $\Delta$-transitive set of order $d$ in $(X,T)$.

\medskip

Clearly, $\pi^{-1}(\pi(x))$ is closed. For  any non-empty open subsets $U_0,U_1,\ldots,U_d$ of $X$ intersecting $\pi^{-1}(\pi(x))$,
since $R \cap\pi^{-1}(\pi(x))$ is dense in $\pi^{-1}(\pi(x))$ and $U_0\cap \pi^{-1}(\pi(x))$ is open in $\pi^{-1}(\pi(x))$, we have that
\begin{equation*}
U_0\cap R \cap \pi^{-1}(\pi(x)) \neq\emptyset.
\end{equation*}
 Choose $x_0\in U_0\cap R \cap \pi^{-1}(\pi(x)) .$
 Then
 \begin{equation*}
  \left (\pi^{-1}(\pi(x))\right)^{d}\subseteq \overline{\O}(x_0^{(d)},\tau_d),
 \end{equation*}
since $\overline{\O}(x_0^{(d)},\tau_d)$ is $\pi^{(d)}$-saturated. It follows that
  \begin{equation*}
   (U_1\times U_2\times\ldots\times U_d)\cap \overline{\O}(x_0^{(d)},\tau_d)\neq\emptyset.
  \end{equation*}
As $x_0\in \Omega_2$,  there exists $n \in\N$ such that
\begin{equation*}
     \tau_d^n(x_0^{(d)})\in U_1\times U_2\times\ldots\times U_d.
\end{equation*}
It follows that
\begin{equation*}
     x_0\in U_0\cap \pi^{-1}(\pi(x)) \cap T^{-n}U_1\cap T^{-2n}U_2\cap\ldots\cap  T^{-dn}U_d.
\end{equation*}
That is,
\begin{equation*}
       N( U_0\cap \pi^{-1}(\pi(x)),U_1,\ldots,U_d)\neq\emptyset.
\end{equation*}
Thus for all $x\in \Omega$, $\pi^{-1}(\pi(x))$ is a $\Delta$-transitive set of order $d$ in $(X,T)$. The proof is completed.
\end{proof}

By Corollary \ref{cor-main-weak} and Corollary \ref{cor-main-distal}, we have the following corollaries.

\begin{thm}
Let $\pi:(X,T)\rightarrow (Y,T)$ be a nontrivial RIC weakly mixing extension of minimal t.d.s. Then there exists a dense $G_\delta$ subset $\Omega$ of $X$ such that for any $x\in\Omega, \pi^{-1}(\pi(x))$ is a $\Delta$-transitive set in $(X,T)$.
\end{thm}

\begin{thm}
Let $(X,T)$ be a distal minimal t.d.s. and let $d \in \N$. If the order of $(X,T)$ is greater than $d$,
there exists a dense $G_\delta$ subset $\Omega$ of $X$ such that for any $x\in\Omega, \pi^{-1}(\pi(x))$ is a $\Delta$-transitive set of order $d+1$ in $(X,T)$.
\end{thm}

Let $(X,T)$ be a dynamical system and $A$ be a closed subset of $X$ with at least two points. We say that $A$ is {\em $\Delta$-weakly mixing} if for every $n \in \N$, $A^n$ is a $\Delta$-transitive subset in $(X^n, T^{(n)})$, that is, for every $n, d \geq 2$ and non-empty open subsets $U_{i,j}$ of $X$ intersecting $A$, $i =1,2,\ldots,n$ and $j=0,1,\ldots,d$, we have
\begin{equation*}
\bigcap_{i=1}^nN(U_{i,0}\cap A, U_{i,1},\ldots,U_{i,d})\neq\emptyset.
\end{equation*}

In \cite{HLYZ}  $\Delta$-weakly mixing sets are systematically studied. Theorem C of \cite{HLYZ} said that for a t.d.s. $(X, T)$ if there exists an ergodic invariant measure $\mu$ such that $(X,\mu,T)$ is not measurable distal, then there exist ``many" $\Delta$-weakly mixing sets in $(X, T)$. Inspired by this result, we have the following question:

\begin{ques}
Let $\pi:(X,T)\rightarrow (Y,T)$ be a nontrivial RIC weakly mixing extension of minimal t.d.s. Does there exist a $\Delta$-weakly mixing set in some fiber $\pi^{-1}(y)$?  Furthermore, is there a dense $G_\delta$ subset $\Omega$ of $X$ such that for any $x\in\Omega, \pi^{-1}(\pi(x))$ is a $\Delta$-weakly mixing set in $(X,T)$?
\end{ques}

\bibliographystyle{plain}

\end{document}